\documentclass[12pt,oneside]{amsart}

\usepackage{amssymb,amsmath}
\usepackage{amsthm}
\usepackage{enumitem}

\numberwithin{equation}{section}
\usepackage{tikz}
\usetikzlibrary{arrows}

\newtheorem{thm}{Theorem}[section]
\newtheorem*{thm_nonumber}{Theorem}
\newtheorem{lem}[thm]{Lemma}
\newtheorem{cor}[thm]{Corollary}

\newtheorem{ques}[thm]{Question}

\theoremstyle{definition}
\newtheorem{defn}[thm]{Definition}
\newtheorem{exmp}[thm]{Example}
\newtheorem{rem}[thm]{Remark}

\numberwithin{equation}{section}

\def\X{{\mathbb X}}
\def\H{{\text {\bf H}}}

\def\P{{\mathbb P}}
\def\ds{\displaystyle}
\def\reg{{\rm reg}}

\def\k{\Bbbk}
\def\deg{{\rm deg}}

\renewcommand{\leq}{\leqslant}
\renewcommand{\geq}{\geqslant}
\renewcommand{\le}{\leqslant}

\DeclareMathOperator{\sdef}{sdefect}

\begin{document}


\title{The symbolic defect of an ideal}

\author{Federico Galetto}
\address{Department of Mathematics and Statistics\\
McMaster University, Hamilton, ON, L8S 4L8}
\email{galettof@math.mcmaster.ca}

\author{Anthony V. Geramita${}^\dag$}
\thanks{${}^\dag$ Deceased, June 22, 2016.}

\author[Y.S. SHIN]{Yong-Su Shin${}^*$}
\address{Department of Mathematics, 
Sungshin Women's University, Seoul, Korea, 136-742}
\email{ysshin@sungshin.ac.kr }
\thanks{${}^*$This research was supported by a grant 
from Sungshin Women's University. Corresponding author}

\author{Adam Van Tuyl${}^{**}$}
\address{Department of Mathematics and Statistics\\
McMaster University, Hamilton, ON, L8S 4L8}
\email{vantuyl@math.mcmaster.ca}
\thanks{${}^{**}$This research was supported 
in part by NSERC Discovery Grant 2014-03898.}


\begin{abstract}
  Let $I$ be a homogeneous ideal of $\k[x_0,\ldots,x_n]$.  To compare
  $I^{(m)}$, the $m$-th symbolic power of $I$, with $I^m$, the regular
  $m$-th power, we introduce the $m$-th symbolic defect of $I$,
  denoted $\sdef(I,m)$.  Precisely, $\sdef(I,m)$ is
  the minimal number of generators of the $R$-module $I^{(m)}/I^m$, or
  equivalently, the minimal number of generators one must add to $I^m$
  to make $I^{(m)}$.  In this paper, we take the first step towards
  understanding the symbolic defect by considering the case that $I$
  is either the defining ideal of a star configuration or the ideal
  associated to a finite set of points in $\mathbb{P}^2$.  We are
  specifically interested in identifying ideals $I$ with
  $\sdef(I,2) = 1$.
\end{abstract}

\maketitle



\section{Introduction}\label{sec:intro}

Let $I$ be a homogeneous ideal of $R = \k[x_0,\ldots,x_n]$.  For any 
positive integer $m$, let $I^{(m)}$ denote the $m$-th symbolic power 
of $I$.  In general, we have $I^m \subseteq I^{(m)}$, but equality may
fail.  During the last decade, there has been interest in the 
so-called ``ideal containment problem,'' that is, for a fixed integer $m$,
find the smallest integer $r$ such that $I^{(r)} \subseteq I^m$.  The 
papers \cite{BH,BH2,D+,DHST,ELS,HH,HoHu,SS} are a small subset 
of the articles on this problem.

In this note, we are also interested in comparing regular and symbolic
powers of ideals, but we wish to investigate a relatively unexplored direction
by measuring the ``difference'' between the two ideals $I^m$ and $I^{(m)}$.
More precisely, because $I^m \subseteq I^{(m)}$, the quotient $I^{(m)}/I^m$
is a finitely generated graded $R$-module.
For any $R$-module $M$, let $\mu(M)$ denote the number of minimal generators
of $M$.  We then define the 
{\it $m$-th symbolic defect of $I$} to be the invariant
\[\sdef(I,m) := \mu(I^{(m)}/I^m),\]
that is, the minimal number of generators of $I^{(m)}/I^m$.  
We will call the sequence 
\[\left\{\sdef(I,m)\right\}_{m \in \mathbb{N}}\]
the {\it symbolic defect sequence}.
Note  that $\sdef(I,m)$ counts the number of generators we need
to add to $I^m$ to make $I^{(m)}$; this invariant can be viewed
as a measure of the failure of $I^m$ to
equal $I^{(m)}$.  For example, $\sdef(I,m) = 0$ if and only if
$I^m = I^{(m)}$.  

We know of only a few papers that have studied the module
$I^{(m)}/I^m$.  This list includes: Arsie and Vatne's paper \cite{AV}
which considers the Hilbert function of $I^{(m)}/I^m$; Huneke's work
\cite{Huneke} which considers $P^{(2)}/P^2$ when $P$ is a height two
prime ideal in a local ring of dimension three; Herzog's paper
\cite{Herzog} which studies the same family of ideals as Huneke using
tools from homological algebra; Herzog and Ulrich's paper \cite{HeUl}
and Vasconcelos's paper \cite{Vas} which also consider a similar
situation to Huneke, but with the assumption that $P$ is generated by
three elements; and Schenzel's work \cite{Shen:2} which describes some
families of prime ideals $P$ of monomial curves with the property that
$P^{(2)}/P^2$ is cyclic (see the comment after \cite[Theorem
2]{Shen:2}).

The introduction of the symbolic defect sequence raises a number of
interesting questions.  For example, how large can
$\sdef(I,m)$ be?  how does $\sdef(I,m)$ compare to
$\sdef(I,m+1)$? and so on.  In some sense, these questions are
difficult since one needs to know both $I^{(m)}$ and $I^m$. To gain
some initial insight into the behavior of the symbolic defect
sequence, in this paper we focus on two cases: (1) $I$ is the defining
ideal of a star configuration, and (2) $I$ is the homogeneous ideal
associated to a set of points in $\mathbb{P}^2$.  In both cases, we
can tap into the larger body of knowledge about these ideals.

To provide some additional focus to our paper, we consider the
following question:

\begin{ques}\label{mainquestion}
  What homogeneous ideals $I$ of $\k[x_0,\ldots,x_n]$ have
  $\sdef(I,2) = 1$?
\end{ques}

\noindent
Because one always has $\sdef(I,1) = 0$, Question
\ref{mainquestion} is in some sense the first non-trivial case to
consider.  Note that when $\sdef(I,2) = 1$, then from an
algebraic point of view, the ideal $I^2$ is almost equal to $I^{(2)}$
except that it is missing exactly one generator.

We now give an outline of the results of this paper.  In Section 2, we
provide the relevant background, and recall some useful tools about
powers of ideals and their symbolic powers.

In Sections 3 through 5, we study $\sdef(I,m)$ when $I$
defines a star configuration.  Note that in this paper, when we refer
to star configurations, the forms that define the star configurations
are forms of any degree, not necessarily linear, which is required in
other papers.  Our main strategy to compute $\sdef(I,m)$ is to
find an ideal $J$ such that $I^{(m)} = J + I^m$, and then to show that
all the minimal generators of $J$ are required.  The recent
techniques using matroid ideals developed by Geramita, Harbourne,
Migliore, and Nagel \cite{GHMN} will play a key role in our proofs.
Our results will imply a similar decomposition found by
Lampa-Baczy{\'n}ska and Malara \cite{LBM} which considers only star
configurations defined using monomial ideals.

In Section 3 we also compute some values of $\sdef(I,m)$ with
$m \geq 3$ for some special families of star configurations.  Section
4 complements Section 3 by showing that under some extra hypotheses,
$\sdef(I,2)=1$ can force a geometric condition.  Specifically,
we show that if $\X$ is a set of points in $\mathbb{P}^2$ with a
linear graded resolution, and if $\sdef(I,2) =1$, then $I$
must be the ideal of a linear star configuration of points in
$\mathbb{P}^2$.  In Section 5 we apply our results of Section 3 to
compute the graded minimal free resolution of $I^{(2)}$ when $I$
defines a star configuration of codimension two in $\mathbb{P}^n$.
This result gives a partial generalization of a result of Geramita,
Harbourne, and Migliore \cite{GHM} (see Remark \ref{generalize}).

In Section 6, we turn our attention to general sets of points in
$\mathbb{P}^2$.  Our main result is a classification of the general
sets of points whose defining ideals $I_\X$ satisfy
$\sdef(I_\X,2) = 1$.

\begin{thm_nonumber}[Theorem \ref{generalpoints}]
  Let $\X$ be a set of $s$ general points in $\mathbb{P}^2$ with
  defining ideal $I_\X$.  Then
  \begin{enumerate}
  \item[$(i)$] $\sdef(I_\X,2) = 0$ if and only if $s = 1,2$ or
    $4$.
  \item[$(ii)$] $\sdef(I_\X,2) = 1$ if and only if
    $s =3, 5, 7$, or $8$.
  \item[$(iii)$] $\sdef(I_\X,2) > 1$ if and only if $s=6$ or
    $s \geq 9$.
  \end{enumerate}
\end{thm_nonumber}

\noindent
Our proof relies on a deep result of Alexander-Hirschowitz \cite{AH:1}
on the Hilbert functions of general double points, and some results of
Catalisano \cite{C}, Harbourne \cite{H1}, and Id\`a \cite{I} on the
graded minimal free resolutions of double points.  We end this paper
with an example to show that the symbolic defect sequence is not
monotonic by computing some values of $\sdef(I_\X,m)$ when
$\X$ is eight general points in $\mathbb{P}^2$ (see Example
\ref{8points}).

\noindent {\bf Acknowledgments.}  Work on this project began in August
2015 when Y.S.~Shin and A.~Van Tuyl visited A.V.~(Tony) Geramita at
his house in Kingston, ON.  F.~Galetto joined this project in late
September of the same year.  Tony Geramita, however, became quite ill
in late December 2015 while in Vancouver, BC, and after a six month
battle with his illness, he passed away on June 22, 2016 in Kingston.
During his illness, we (the remaining co-authors) kept Tony up-to-date
of the status on the project, and when his health permitted, he would
contribute ideas to this paper.  He was looking forward to returning
to Kingston, and turning his attention to this paper.  Unfortunately,
this was not to be.  Although Tony was not able to see the final
version of this paper, we feel that his contributions warrant an
authorship.  Those familiar with Tony's work will hopefully
recognize Tony's interests and contributions to the
topics in this paper.  Tony is greatly missed.

We would also like to thank Brian Harbourne and Alexandra Seceleanu
for their helpful comments.  Part of this paper was written at the
Fields Institute; the authors thank the institute for its hospitality.
Finally, we would like to thank the referees for their helpful suggestions
and corrections.


\section{Background results}\label{sec:background}

We review the required background.  We continue to use the notation of
the introduction.  Let $I$ be a homogeneous ideal of
$R = \k[x_0,\ldots,x_n]$.  The {\it $m$-th symbolic power of $I$},
denoted $I^{(m)}$, is defined to be
\[I^{(m)} = \bigcap_{P \in {\rm Ass}(I)} (I^mR_P \cap R)\] where
${\rm Ass}(I)$ denotes the set of associated primes of $I$ and $R_P$
is the ring $R$ localized at the prime ideal $P$.
\begin{rem}
  There is some ambiguity in the literature concerning the notion of
  symbolic powers. The intersection in the definition is sometimes
  taken over all associated primes and sometimes just over the minimal
  primes of $I$. In general, these two possible definitions yield
  different results. However, they agree in the case of radical
  ideals.
\end{rem}
In general, $I^m \subseteq I^{(m)}$, but the reverse containment may
fail.  If $\sdef(I,m) = s$, then there exist $s$ homogeneous
forms $F_1,\ldots,F_s$ of $R$ such that
\[I^{(m)}/I^m = \langle F_1 + I^m, \ldots, F_s + I^m \rangle \subseteq
  R/I^m.\] It follows that
$I^{(m)} = \langle F_1,\ldots,F_s \rangle + I^m.$ Note that the ideal
$\langle F_1,\ldots,F_s \rangle$ is not unique. Indeed, if
$G_1,\ldots,G_s$ is another set of coset representatives such that
$I^{(m)}/I^m = \langle G_1 + I^m, \ldots, G_s + I^m \rangle$, we still
have $I^{(m)} = \langle G_1,\ldots,G_s \rangle + I^m$, but
$\langle F_1,\ldots, F_s \rangle$ and $\langle G_1,\ldots,G_s \rangle$
may be different ideals.

We state some simple facts about $\sdef(I,m)$.

\begin{lem} \label{sdefectlemma} Let $I$ be a homogeneous radical
  ideal of $R$.
  \begin{enumerate}[label=(\roman*)]
  \item $\sdef(I,1) = 0$.
  \item If $I$ is a complete intersection, then
    $\sdef(I,m) = 0$ for all $m \geq 1$.
  \end{enumerate}
\end{lem}

\begin{proof}
  $(i)$ This fact is trivial.  $(ii)$ This result follows from
  Zariski-Samuel \cite[Appendix 6, Lemma 5]{ZS}.
\end{proof}

Recall that $I$ is a {\it generic complete intersection} if the
localization of $I$ at any minimal associated prime of $I$ is a
complete intersection.  A result of \cite{Cetal,SV,Wey} will prove
useful:

\begin{thm}[{\cite[Corollary 2.6]{Cetal},\cite{SV}\cite{Wey}}]\label{resix2}
  Let $I$ be a homogeneous ideal of $\k[x_0,\ldots,x_n]$ that is
  perfect, codimension two, and a generic complete intersection.  If
  \[0 \longrightarrow F \longrightarrow G \longrightarrow I
    \longrightarrow 0\] is a graded minimal free resolution of $I$,
  then
  \[0 \longrightarrow \bigwedge^2 F \longrightarrow F\otimes G \rightarrow {\rm Sym}^2 G \longrightarrow I^2
    \longrightarrow 0\] is a graded minimal free resolution of $I^2$.
\end{thm}

  \begin{rem}
    Weyman's paper \cite{Wey} gives the resolution of
    ${\rm Sym}^2(I)$.  As shown in \cite{Cetal,SV}, the hypotheses on
    $I$ imply that ${\rm Sym}^2(I) \cong I^2$.
  \end{rem}

Many of our arguments make use of Hilbert functions.  The {\it Hilbert
  function} of $R/I$, denoted $\H_{R/I}$, is the numerical function
$\H_{R/I}:\mathbb{N} \rightarrow \mathbb{N}$ defined by
\[\H_{R/I}(i) := \dim_\k R_i - \dim_\k I_i\]
where $R_i$, respectively $I_i$, denotes the $i$-th graded component
of $R$, respectively $I$.

Our primary focus is to understand $\sdef(I,m)$ when $I$
defines either a star configuration or a set of points in
$\mathbb{P}^2$.  In the next section, we introduce star configurations
in more detail.  For now, we review the relevant background about sets
of points in $\mathbb{P}^2$.

Let $\X = \{P_1,\ldots,P_s\}$ be a set of distinct points in
$\mathbb{P}^2$.  If $I_{P_i}$ is the ideal associated to $P_i$ in
$R = \k[x_0,x_1,x_2]$, then the homogeneous ideal associated to $\X$
is the ideal $I_\X = I_{P_1} \cap \cdots \cap I_{P_s}$.
  The next lemma allows us to describe $I_\X^{(m)}$; although this
  result is well-known, we have included a proof for completeness.

\begin{lem} Let $\X = \{P_1,\ldots,P_s\} \subseteq \mathbb{P}^2$ be a
  set of $s$ distinct points with associated ideal
  $I_\X = I_{P_1} \cap \cdots \cap I_{P_s}$.  Then for all $m \geq 1$,
  $I^{(m)}_\X = I^m_{P_1} \cap \cdots \cap I^m_{P_s}$.
\end{lem}

\begin{proof}
  The associated primes of $I_\X$ are the ideals $I_{P_i}$ with
  $i=1,\ldots,s$.  Because localization commutes with products, we
  have
  \[I^m_\X R_{I_{P_i}} = (I_\X R_{I_{P_i}})^m = (I_{P_i}R_{P_i})^m =
    I_{P_i}^mR_{P_i}.\] Note that the second equality follows from the
  fact that $I_{P_i}$ is the only associated prime of $I_\X$ contained
  in $I_{P_i}$.  Since $I_{P_i}^mR_{P_i} \cap R = I_{P_i}^m$, the
  result follows.
\end{proof}

For sets of points in $\mathbb{P}^2$, the symbolic defect sequence
will either be all zeroes, or all values of the sequence, except the
first, will be nonzero.  Moreover, we can completely classify when the
symbolic defect sequence is all zeroes.

\begin{thm}\label{completeintersection}
  Let $\X \subseteq \mathbb{P}^2$ be any set of points.  Then the
  following are equivalent:
  \begin{enumerate}
  \item[$(i)$] $I_\X$ is a complete intersection.
  \item[$(ii)$] $\sdef(I_\X,m) = 0$ for all $m \geq 1$.
  \item[$(iii)$] $\sdef(I_\X,m) = 0$ for some $m \geq 2$.
  \end{enumerate}
\end{thm}

\begin{proof}
  Lemma \ref{sdefectlemma} shows $(i) \Rightarrow (ii)$, and
  $(ii) \Rightarrow (iii)$ is immediate.  For $(iii) \Rightarrow (i)$,
  it was noted in \cite[Remark 2.12(i)]{Cetal} that when $\X$ is not a
  complete intersection of points in $\mathbb{P}^2$, then
  $I_\X^m \neq I_\X^{(m)}$ for all $m \geq 2$.  This also
    follows from \cite[Theorem 2.8]{HU} or \cite[Corollary
    2.5]{Huneke}.
\end{proof}


\section{Symbolic squares of star configurations}
\label{sec:symbolic-square-star}

In this section, we will consider $\sdef(I,2)$ when $I$ defines a star
configuration.  In fact, we prove a stronger result by finding an
ideal $J$ such that $I^{(2)} = J + I^2$.  It is interesting to note
that the ideal $J$ will also be a star configuration.  For
completeness, we begin with the relevant background on star
configurations.

\begin{defn}
  Let $n$, $c$ and $s$ be positive integers with
  $1\leqslant c\leqslant \min\{n,s\}$. Let
  $\mathcal{F} = \{F_1,\ldots,F_s\}$ be a set of forms in
  $R=\k[x_0,x_1,\ldots,x_n]$ with the property that all subsets of
  $\mathcal{F}$ of cardinality $c+1$ are regular sequences in
  $R$. Define an ideal of $R$ by setting
  \begin{equation*}
    I_{c,\mathcal{F}} = \bigcap_{1\leqslant i_1<\ldots<i_{c}\leqslant s}
    \langle F_{i_1},\ldots,F_{i_c} \rangle.
  \end{equation*}
  The vanishing locus of $I_{c,\mathcal{F}}$ in $\mathbb{P}^n$ is
  called a \emph{star configuration}.

  When the forms $F_1,\ldots,F_s$ are all linear, 
we will typically use $L_i$ in place of $F_i$ and write 
  $\mathcal{L}=\{L_1,\ldots,L_s\}$ in place of
  $\mathcal{F} = \{F_1,\ldots,F_s\}$, and we will call the vanishing
  locus of $I_{c,\mathcal{L}}$ a \emph{linear star configuration}.
\end{defn}

\begin{rem}
  A.V.~Geramita is attributed with first coining the term star
  configuration to describe the variety defined by
  $I_{c,\mathcal{F}}$.  The name is inspired by the fact that when
  $n=c=2$, and $s=5$, the placement of the five lines
  $\mathcal{L} = \{L_1,\ldots,L_5\}$ that define a linear star
  configuration resembles a star.  In this case, the locus of
  $I_{c,\mathcal{L}}$ is a set of 10 points corresponding to the
  intersections between these lines.  It should be noted that linear
  star configurations were classically called $l$-laterals (e.g.  see
  \cite{Do}).  On the other hand, our more general definition follows
  \cite{GHMN}, where the geometric objects are called hypersurface
  configurations.  This more general definition of star configurations
  evolved through a series of papers (see \cite{AS:1,PS:1,GHMN}); in
  particular, the codimension 2 case was studied before the general
  case.  Star configurations have been shown to have many nice
  algebraic properties, but at the same time, can be used to exhibit
  extremal properties.  The references \cite{BH,BH2,CGVT, GHM,GMS:1}
  form a small sample of papers that have studied the ideals
  $I_{c,\mathcal{F}}$.
\end{rem}
\begin{rem}
  Geometrically, the vanishing locus in $\mathbb{P}^n$ of the ideal
  $\langle F_{i_1},\ldots,F_{i_c} \rangle$ is a complete intersection
  of codimension $c$ obtained by intersecting the hypersurfaces
  defined by the forms $F_{i_1},\ldots,F_{i_c}$. A star configuration
  is then a union of such complete intersections.
\end{rem}

\begin{rem}
  While the definition of a star configuration makes sense for
  $s<n+1$, such cases are less interesting (cf. \cite[Remark
  2.2]{GHM}). Therefore we will always assume that $s\geqslant n+1$.
\end{rem}

\begin{thm}
  \label{pro:star-generators}
  Let $I_{c,\mathcal{F}}$ be the defining ideal of a star
  configuration in $\mathbb{P}^n$, with
  $\mathcal{F}=\{F_1,\ldots,F_s\}$. Then
  \begin{equation*}
    \{F_{i_1} \cdots F_{i_{s-c+1}} \mid
    1 \leqslant i_1 < \ldots < i_{s-c+1} \leqslant s\}
  \end{equation*}
  is a minimal generating set of $I_{c,\mathcal{F}}$.
\end{thm}

\begin{proof}
  See \cite[Theorem 2.3]{PS:1} for generation (see also
    \cite[Proposition 2.3 (4)]{GHMN}) and \cite[Corollary 3.5]{PS:1}
  for minimality.
\end{proof}

We will make use of the following decomposition of the $m$-th symbolic
power; this follows from \cite[Theorem 3.6 (i)]{GHMN}.

\begin{thm}
  \label{pro:symbolic-powers-star}
  Let $I_{c,\mathcal{F}}$ be the defining ideal of a star
  configuration in $\mathbb{P}^n$, with
  $\mathcal{F}=\{F_1,\ldots,F_s\}$. For all $m\geqslant 1$, we have
  \begin{equation*}
    I_{c,\mathcal{F}}^{(m)} = \bigcap_{1\leqslant i_1<\ldots<i_{c}\leqslant s}
    \langle F_{i_1},\ldots,F_{i_c} \rangle^m.
  \end{equation*}
\end{thm}

We will first consider the case of a linear star configuration
$I_{c,\mathcal{L}}$ in $\mathbb{P}^n$, with
$\mathcal{L}=\{L_1,\ldots,L_s\}$ when $|\mathcal{L}|=n+1$.  In this
context, we can reduce to the case of monomial ideals.  Then,
following \cite{GHMN}, we will apply our results to obtain
corresponding statements for arbitrary star configurations.

\subsection{The monomial case}
\label{sec:monomial-case}

Let $I_{c,\mathcal{L}}$ be the defining ideal of a linear star
configuration in $\mathbb{P}^n$, with
$\mathcal{L}=\{L_1,\ldots,L_s\}$. Suppose that
$|\mathcal{L}|=n+1$. Then, up to a change of variables, we may assume
that the hyperplanes forming the star configuration are defined by the
coordinate functions $x_0,x_1,\ldots,x_n$. By Theorem
\ref{pro:symbolic-powers-star}, we have
\begin{equation*}
  I^{(m)}_{c,\mathcal{L}} = \bigcap_{0\leqslant i_1 < \ldots < i_c \leqslant n}
  \langle x_{i_1},\ldots,x_{i_c}\rangle^m.
\end{equation*}
Clearly, $I_{c,\mathcal{L}}$ and its symbolic powers are monomial
ideals. A monomial $p = x_0^{a_0} x_1^{a_1} \cdots x_n^{a_n}$ belongs
to $I^{(m)}_{c,\mathcal{L}}$ if and only if it satisfies the condition
\begin{equation}
  \label{eq:1}
    a_{i_1} + a_{i_2} + \cdots + a_{i_c} \geqslant m \text{ for all } 
0\leqslant i_1 < \cdots < i_c \leqslant n.
\end{equation}
Let $\operatorname{Supp} (p)$ denote the support of $p$, i.e.,
$\operatorname{Supp} (p) = \{ x_i \mid x_i \text{ divides } p\}$.

We are now able to describe an ideal $M$ with the property that
$I^{(m)}_{c,\mathcal{L}} = I^m_{c,\mathcal{L}} + M$.

\begin{thm}
  \label{thm:symbolic-monomial}
  Let $\mathcal{L} = \{x_0,\ldots,x_n\}$.
  Then $I^{(m)}_{c,\mathcal{L}} = I^m_{c,\mathcal{L}} + M$, where $M$
  is the ideal generated by all monomials satisfying equation
  \eqref{eq:1} whose support has cardinality at least $n-c+3$.
\end{thm}

\begin{proof}
  Clearly $I^{(m)}_{c,\mathcal{L}} \supseteq I^m_{c,\mathcal{L}} + M$.
  To show the other containment, consider a monomial
  $p = x_0^{a_0} x_1^{a_1} \ldots x_n^{a_n} \in
  I^{(m)}_{c,\mathcal{L}}$.  Since $p\in I^{(m)}_{c,\mathcal{L}}$, we
  have $p\in I_{c,\mathcal{L}}$. Then
  $|\operatorname{Supp} (p)| \geqslant n-c+2$ by Theorem
  \ref{pro:star-generators}.

  If $|\operatorname{Supp} (p)| = n-c+2$, then the complement of
  $\operatorname{Supp} (p)$ in $\{x_0,x_1,\ldots,x_n\}$ has
  cardinality $c-1$. Therefore we can write
  \begin{equation*}
    \{x_0,x_1,\ldots,x_n\} \setminus \operatorname{Supp}(p)
    = \{x_{j_1},\ldots,x_{j_{c-1}}\}.
  \end{equation*}
  For each $x_i \in \operatorname{Supp}(p)$, equation \eqref{eq:1}
  implies that
  \begin{equation*}
    a_i = a_i + a_{j_1} + \ldots + a_{j_{c-1}} \geqslant m.
  \end{equation*}
  Thus $p$ is a multiple of
  \begin{equation*}
    \prod_{x_i \in \operatorname{Supp} (p)} x_i^m =
    \Bigg(\prod_{x_i \in \operatorname{Supp} (p)} x_i\Bigg)^m
  \end{equation*}
  which is the $m$-th power of a generator of $I_{c,\mathcal{L}}$ by
  Theorem \ref{pro:star-generators}. Therefore
  $p\in I^m_{c,\mathcal{L}}$.

  On the other hand, if $|\operatorname{Supp} (p)| \geqslant n-c+3$,
  then $p \in M$ by definition.
\end{proof}

For $m=2$ and $m=3$, we can improve upon the statement of Theorem
\ref{thm:symbolic-monomial}.

\begin{cor}
  \label{cor:symb-square-monomial}
  Let $\mathcal{L} = \{x_0,\ldots,x_n\}$.
  We have
  $I^{(2)}_{c,\mathcal{L}} = I_{c-1,\mathcal{L}} +
  I^2_{c,\mathcal{L}}$.
\end{cor}
\begin{proof}
  By \cite[Lemma 2.13]{GHM}, we have
  $I_{c-1,\mathcal{L}} \subseteq I^{(2)}_{c,\mathcal{L}}$, which
  implies the containment
  $I^{(2)}_{c,\mathcal{L}} \supseteq I_{c-1,\mathcal{L}} +
  I^2_{c,\mathcal{L}}$ (these containments hold for any
    linear star configuration ideal, not just a monomial
  star configuration ideal).  To prove the other
  containment, we use the fact that our ideals are monomial ideals.

  Consider a monomial
  $p = x_0^{a_0} x_1^{a_1} \ldots x_n^{a_n} \in
  I^{(2)}_{c,\mathcal{L}}$.  As observed in the proof of Theorem
  \ref{thm:symbolic-monomial},
  $|\operatorname{Supp} (p)| \geqslant n-c+2$ and, in the case of
  equality, $p\in I^{2}_{c,\mathcal{L}}$.  Assume
  $|\operatorname{Supp} (p)| \geqslant n-c+3$.  Then $p$ is divisible
  by one of the generators of $I_{c-1,\mathcal{L}}$ described in
  Theorem \ref{pro:star-generators}. Therefore
  $p \in I_{c-1,\mathcal{L}}$.
\end{proof}

\begin{rem}
  The above result was first proved in \cite[Corollary 3.7, Corollary
  4.5]{LBM} in the special cases that $n=c=2$, and $n=c=3$.  The above
  statement is also mentioned in \cite[Remark 4.6]{LBM}, but no proof
  is given.
\end{rem}

\begin{cor}
  \label{cor:symb-cube-monomial}
  Let $\mathcal{L} = \{x_0,\ldots,x_n\}$.
  If $c \geq 3$, we have
  $I^{(3)}_{c,\mathcal{L}} = I_{c-2,\mathcal{L}} + I_{c-1,\mathcal{L}}
  I_{c,\mathcal{L}} + I^3_{c,\mathcal{L}}$.
\end{cor}
\begin{proof}
  We require $c \geq 3$ so that the ideals on the right
    hand side are defined.  We first show that
    $I_{{c-2},\mathcal{L}} \subseteq I^{(3)}_{c,\mathcal{L}}$.  Recall
    that
$$
I_{c-2,\mathcal{L}} = \langle x_{i_1} \cdots x_{i_{n-c+4}} \mid 0\le
i_1<\cdots < i_{n-c+4}\le n \rangle.
$$
Consider any subset $A = \{x_{i_1},\ldots,x_{i_c}\}$ of
$\{x_0,x_1,\dots,x_n\}$ with $|A|=c$, and consider any generator
$m = x_{i_1} \cdots x_{i_{n-c+4}}$ of $I_{c-2,\mathcal{L}}$.  Then at
least three of the variables of $A$, say $x_i,x_j,$ and $x_k$, appear
in ${\rm Supp}(m) = \{ x_{i_1}, \ldots, x_{i_{n-c+4}}\}$.  Because
$x_ix_jx_k \in \langle x_{i_1},\ldots,x_{i_c} \rangle^3$, this means
that $m \in \langle x_{i_1},\ldots,x_{i_c} \rangle^3$.  But this
implies that every generator $m$ of $I_{{c-2},\mathcal{L}}$ satisfies
\[m \in \left(\bigcap_{0 \leq i_1 < \cdots < i_c \leq n} \langle
    x_{i_1},\ldots,x_{i_c} \rangle^3 \right)=
  I_{c,\mathcal{L}}^{(3)}.\] In other words,
$I_{{c-2},\mathcal{L}} \subseteq I^{(3)}_{c,\mathcal{L}}$.

By \cite[Lemma 2.13]{GHM}, we have
$I_{{c-1},\mathcal{L}} \subseteq I_{c,\mathcal{L}}^{(2)}$.  This
result allows us to conclude that
\begin{equation*}
  I_{c-1,\mathcal{L}} I_{c,\mathcal{L}} \subseteq
  I^{(2)}_{c,\mathcal{L}} I_{c,\mathcal{L}}
  \subseteq I^{(3)}_{c,\mathcal{L}}.
\end{equation*}
Therefore we have the containment
$I^{(3)}_{c,\mathcal{L}} \supseteq I_{c-2,\mathcal{L}} +
I_{c-1,\mathcal{L}} I_{c,\mathcal{L}} + I^3_{c,\mathcal{L}}$.  To
prove the other containment, we again exploit the fact that our ideals
are all monomial.

Consider a monomial
$p = x_0^{a_0} x_1^{a_1} \ldots x_n^{a_n} \in
I^{(3)}_{c,\mathcal{L}}$.  By Theorem \ref{thm:symbolic-monomial},
$|\operatorname{Supp} (p)| \geqslant n-c+2$ and, in the case of
equality, $p\in I^{3}_{c,\mathcal{L}}$.  Let
$|\operatorname{Supp} (p)| =n-c+3$. In this case, the complement of
$\operatorname{Supp} (p)$ in $\{x_0,x_1,\ldots,x_n\}$ has cardinality
$c-2$, so we can write
\begin{equation*}
  \{x_0,x_1,\ldots,x_n\} \setminus \operatorname{Supp}(p)
  = \{x_{j_1},\ldots,x_{j_{c-2}}\}.
\end{equation*}
For each pair $x_{i_1},x_{i_2} \in \operatorname{Supp}(p)$, equation
\eqref{eq:1} implies that
\begin{equation*}
  a_{i_1} + a_{i_2} = a_{i_1} + a_{i_2} + a_{j_1} + \ldots + a_{j_{c-2}} \geqslant 3.
\end{equation*}
Thus either $a_{i_1} \geqslant 2$ or $a_{i_2} \geqslant 2$. Repeating
the same argument for all pairs $x_{i_1},x_{i_2}$ in
$\operatorname{Supp}(p)$, it follows that there are $n-c+2$ elements
$x_h \in \operatorname{Supp} (p)$ such that $x_h^2 \mid p$. Hence $p$
is divisible by a monomial of the form
\begin{equation*}
  x_{k_0} x_{k_1}^2 \ldots x_{k_{n-c+2}}^2 =
  (x_{k_0} x_{k_1} \ldots x_{k_{n-c+2}})
  (x_{k_1} \ldots x_{k_{n-c+2}}),
\end{equation*}
and therefore $p \in I_{c-1,\mathcal{L}} I_{c,\mathcal{L}}$ by Theorem
\ref{pro:star-generators}.  As in the previous proof,
if $|\operatorname{Supp} (p)| \geqslant n-c+4$, then $p$ is divisible
by a generator of $I_{c-2,\mathcal{L}}$, which completes the proof.
\end{proof}

\begin{thm}
  \label{cor:monomial-sdefect}
  Let $\mathcal{L} = \{x_0,\ldots,x_n\}$.  We have
  $\sdef (I_{c,\mathcal{L}},m) =1$ if and only if $c=m =2$.
\end{thm}
\begin{proof}
  Let $c=m=2$. By Theorem \ref{pro:star-generators},
  $I_{c-1,\mathcal{L}}=I_{1,\mathcal{L}}=\langle x_0 x_1
    \cdots x_n \rangle$ is a principal ideal generated in degree
    $n+1$. In contrast, $I_{c,\mathcal{L}}^2$ is generated in degree
  $n^2$. Therefore, the equality
  $I^{(2)}_{c,\mathcal{L}} = I_{c-1,\mathcal{L}} +
  I^2_{c,\mathcal{L}}$ of Corollary \ref{cor:symb-square-monomial},
  implies that $I^{(2)}_{c,\mathcal{L}}/I^2_{c,\mathcal{L}}$ has a
  single minimal generator. Thus $\sdef (I_{c,\mathcal{L}},m) =1$.

  Conversely, assume $\sdef (I_{c,\mathcal{L}},m) =1$. By Theorem
  \ref{thm:symbolic-monomial},
  $I_{c,\mathcal{L}}^{(m)} = I_{c,\mathcal{L}}^m + M$, where $M$ is
  the monomial ideal generated by all monomials satisfying equation
  \eqref{eq:1} whose support has cardinality at least $n-c+3$. Since
  $\sdef (I_{c,\mathcal{L}},m) =1$, we deduce $M\neq 0$. Given any
  monomial $p \in M$, we must have
  \begin{equation*}
    n+1 \geqslant |\operatorname{Supp} (p) | \geqslant n-c+3.
  \end{equation*}
  This implies $c\geqslant 2$. For any choice of indices
  $0\leqslant i_{1} < \cdots < i_{n-c+3} \leqslant n$, the monomial
  \begin{equation}
    \label{eq:3}
    p = x_{i_1} x_{i_2}^{m -1} x_{i_3}^{m -1} \cdots x_{i_{n-c+3}}^{m-1}
  \end{equation}
  satisfies the condition in equation \eqref{eq:1}, and therefore
  $p\in M$. We claim that $p$ is a minimal generator of $M$. If it was
  not, then we could divide $p$ by a variable in its support and
  obtain a new monomial still in $M$. However, if we divide $p$ by any
  variable in its support, we either obtain a monomial whose support
  has cardinality less than $n-c+3$ or one that violates equation
  \eqref{eq:1}. Thus the claim holds.  Note also that the degree of
  $p$ is $(m -1)(n-c+2)+1$, and this is strictly smaller than the
  degree of a minimal generator of $I_{c,\mathcal{L}}^m$, i.e.,
  $m (n-c+2)$. It follows that the residue class of $p$ can be taken
  as a minimal generator of
  $I_{c,\mathcal{L}}^{(m)} / I_{c,\mathcal{L}}^m$. Hence each monomial
  of the same form as $p$ contributes 1 to
  $\sdef (I_{c,\mathcal{L}}, m)$.  Now, if $c>2$ or $m>2$, the freedom
  in the choice of the indices $i_{1}, \ldots, i_{n-c+3}$ implies that
  $\sdef (I_{c,\mathcal{L}}, m) > 1$.
\end{proof}

\subsection{The general case}
\label{sec:general-case}

To extend the results of the monomial case to arbitrary star
configurations, we recall a powerful theorem of Geramita, Harbourne,
Migliore, and Nagel \cite[Theorem 3.6 (i)]{GHMN}.

\begin{thm}
  \label{thm:GHMN-main}
  Let $I_{c,\mathcal{F}}$ be the defining ideal of a star
  configuration in $\mathbb{P}^n$, with
  $\mathcal{F}=\{F_1,\ldots,F_s\}\subseteq
  R=k[x_0,x_1,\ldots,x_n]$. Let $S=k[y_1,\ldots,y_s]$ and define a
  ring homomorphism $\varphi \colon S\to R$ by setting
  $\varphi (y_i) = F_i$ for $1\leq i\leq s$. If $I$ is an ideal of
  $S$, then we write $\varphi_* (I)$ to the denote the ideal of $R$
  generated by $\varphi (I)$. Let $\mathcal{L} =
  \{y_1,\ldots,y_s\}$. Then, for each positive integer $m$, we have
  \begin{equation*}
    I_{c,\mathcal{F}}^{(m)} = \varphi_* (I_{c,\mathcal{L}})^{(m)} =
    \varphi_* (I_{c,\mathcal{L}}^{(m)}).
  \end{equation*}
\end{thm}

Since the operator $\varphi_*$ commutes with ideal sums and products,
Theorem \ref{thm:GHMN-main} applied to our results from the previous
section gives the following more general statements.

\begin{thm}
  Let $I_{c,\mathcal{F}}$ be the defining ideal of a star
  configuration in $\mathbb{P}^n$, with
  $\mathcal{F}=\{F_1,\ldots,F_s\}$.  Then
  $I^{(m)}_{c,\mathcal{F}} = I^m_{c,\mathcal{F}} + M$, where $M$ is
  the ideal generated by all products $F_1^{a_1} \cdots F_s^{a_s}$
  such that:
  \begin{enumerate}
  \item $|\{i \mid a_i > 0\}| \geq s-c+2$;
  \item
    $\forall 0\leqslant i_1 < \ldots < i_c \leqslant n,\ a_{i_1} +
    a_{i_2} + \ldots + a_{i_c} \geqslant m$.
  \end{enumerate}
\end{thm}

\begin{cor}
  \label{cor:general-symb-square}
  We have
  $I^{(2)}_{c,\mathcal{F}} = I_{c-1,\mathcal{F}} +
  I^2_{c,\mathcal{F}}$.
\end{cor}

\begin{cor} We have
  $\sdef(I_{c,\mathcal{F}},2) \leq
    \binom{s}{c-2}$.  Furthermore, if
  $\mathcal{F} = \mathcal{L} = \{L_1,\ldots,L_s\}$, that is, if
  $I_{c,\mathcal{L}}$, is a linear star configuration, then
  $\sdef(I_{c,\mathcal{L}},2) = \binom{s}{c-2}$.
\end{cor}

\begin{proof}
  By Corollary \ref{cor:general-symb-square},
  $I^{(2)}_{c,\mathcal{F}} = I_{c-1,\mathcal{F}} +
  I^2_{c,\mathcal{F}}$.  By Theorem \ref{pro:star-generators}, the
  ideal $I_{c-1,\mathcal{L}}$ is generated by
  $\binom{s}{s-c+2} = \binom{s}{c-2}$ minimal
    generators, so we need to add at most $\binom{s}{c-2}$ generators
  to $I^2_{c,\mathcal{F}}$ to generate $I^{(2)}_{c,\mathcal{L}}$.

  If $\mathcal{F} = \mathcal{L}$, by Theorem \ref{pro:star-generators}
  $I^2_{c,\mathcal{L}}$ is generated by forms of degree $2(s-c+1)$.
  On the other hand, again by Theorem \ref{pro:star-generators}, the
  ideal $I_{c-1,\mathcal{L}}$ is generated by generators of degree
  $s-c+2$.  Since $s-c+2 < 2(s-c+1)$, all the generators of
  $I_{c-1,\mathcal{L}}$ need to be added to $I^2_{c,\mathcal{L}}$ to
  generate $I_{c,\mathcal{L}}^{(2)}$, i.e., none of them are
  redundant.
\end{proof}

\begin{rem}
  In the above proof, we appealed to the degrees of the elements of
  $\mathcal{L}$ to justify why all the generators of
  $I_{c-1,\mathcal{L}}$ are required.  In the general case, it may
  happen that some of the minimal generators of $I_{c-1,\mathcal{F}}$
  have degree larger than a minimal generator of
  $I_{c,\mathcal{F}}^2$, thus preventing us from generalizing this
  argument.
\end{rem}

The following are also immediate consequences of results from the
previous section.

\begin{cor}
  We have
  $I^{(3)}_{c,\mathcal{F}} = I_{c-2,\mathcal{F}} + I_{c-1,\mathcal{F}}
  I_{c,\mathcal{F}} + I^3_{c,\mathcal{F}}$.  In particular,
  \[\sdef(I_{c,\mathcal{F}},3) \leq \binom{s}{c-3} +
    \binom{s}{c-2}\binom{s}{c-1}.\]
\end{cor}

\begin{thm}\label{sdefect-one-classify}
  We have $\sdef (I_{c,\mathcal{F}},m) =1$ if and only if $c=m =2$.
\end{thm}

\subsection{Powers of codimension two linear star configurations}
We round out this section by considering the higher $m$-th symbolic
powers of the linear star configuration $I_{2,\mathcal{L}}$ in
$\mathbb{P}^2$.  Note that in this case the linear star configuration
defines a collection of points in $\mathbb{P}^2$.  By applying \cite[Corollary 3.9]{HH} 
 of  Harbourne and Huneke (and see also
\cite[Example 3.9]{Cetal} for additional details), we have the
following relationship between the regular and symbolic powers of
$I_{2,\mathcal{L}}$ in $\mathbb{P}^2$.

\begin{thm}\label{higherpowers}
  Suppose that $I_{2,\mathcal{L}}$ defines a linear star configuration
  in $\mathbb{P}^2$.  Then
  \[I_{2,\mathcal{L}}^{(2m)} = (I_{2,\mathcal{L}}^{(2)})^m ~~\mbox{
      for all $m \geq 1$.}\]
\end{thm}

We can then derive bounds on some of the values of the symbolic defect
sequence.

\begin{thm}
  Suppose that $I_{2,\mathcal{L}}$ defines a linear star configuration
  in $\mathbb{P}^2$.  Then
  \[\sdef(I_{2,\mathcal{L}},2m) \leq 1+|\mathcal{L}|(m-1) ~~\mbox{for
      all $m\geq 1$.}\]
\end{thm}

\begin{proof}
  Suppose that $\mathcal{L} = \{L_1,\ldots,L_s\}$.  By Corollary
  \ref{cor:general-symb-square} we have
  \[I_{2,\mathcal{L}}^{(2)} = \langle L_1\cdots L_s \rangle +
    I_{2,\mathcal{L}}^2\] since
  $I_{1,\mathcal{L}} = \langle L_1\cdots L_s \rangle$.  Let
  $L = L_1\cdots L_s$.  It then follows by Theorem \ref{higherpowers}
  that
  \begin{eqnarray*}
    I_{2,\mathcal{L}}^{(2m)} & = & \left[\langle L \rangle + I_{2,\mathcal{L}}^2\right]^m \\
                             & = & \langle L \rangle^m + \langle L \rangle^{m-1}I_{2,\mathcal{L}}^2
                                   +  \langle L \rangle^{m-2}I_{2,\mathcal{L}}^4
                                   + \cdots +  \langle L \rangle^{1}I_{2,\mathcal{L}}^{2m-2} +  I_{2,\mathcal{L}}^{2m}.
  \end{eqnarray*}
  Since $I_{2,\mathcal{L}}$ is generated by forms of degree $(s-1)$,
  we can use a a degree argument to show that none of the generators
  of
  $\langle L \rangle^m + \langle L \rangle^{m-1}I_{2,\mathcal{L}}^2 +
  \langle L \rangle^{m-2}I_{2,\mathcal{L}}^4 + \cdots + \langle L
  \rangle^{1}I_{2,\mathcal{L}}^{2m-2}$ belong to
  $I_{2,\mathcal{L}}^{2m}$.

  Define
    $J_{2a} = \langle \frac{L^{2a}}{L_i^{2a}} ~|~ i=1,\ldots,s
    \rangle$ for $a = 1,\ldots,m-1$.  We claim that for
    $1\le a \le m-1$,
    \begin{multline*}
      \langle L \rangle^m + \langle L \rangle^{m-1}I_{2,\mathcal{L}}^2
      + \cdots + \langle L \rangle^{m-a+1}I_{2,\mathcal{L}}^{2(a-1)} +
      \langle L \rangle^{m-a}I_{2,\mathcal{L}}^{2a}
      \\
      = \langle L \rangle^m + \langle L \rangle^{m-1}J_2 + \cdots +
      \langle L \rangle^{m-a+1}J_{a-1} + \langle L
      \rangle^{m-a}I_{2,\mathcal{L}}^{2a}.
    \end{multline*}
    Indeed, the ideal on the right is contained in the ideal on the
    left because each generator of $J_{2a}$ is a generator of
    $I_{2,\mathcal{L}}^{2a}$.

    For the reverse containment, we do induction on $a$.  It is
    straightforward to check that
    $\langle L \rangle^m + \langle L \rangle^{m-1}I_{2,\mathcal{L}}^2
    = \langle L \rangle^m + \langle L \rangle^{m-1}J_2$ for the base
    case.  Assume now that {$2 \leq a \le m-1$}.  By induction on $a$,
    \begin{multline*}
      \langle L \rangle^m + \langle L \rangle^{m-1}I_{2,\mathcal{L}}^2
      + \cdots + \langle L \rangle^{m-a+1}I_{2,\mathcal{L}}^{2(a-1)} =
      \langle L \rangle^m + \langle L \rangle^{m-1}J_2 + \cdots +
      \langle L \rangle^{m-(a-1)}J_{2(a-1)}.
    \end{multline*}
    To finish the proof of the claim, we need to show that
    \[
      \langle L \rangle^{m-a}I_{2,\mathcal{L}}^{2a} \subseteq \langle
      L \rangle^m + \langle L \rangle^{m-1}J_2 + \cdots + \langle L
      \rangle^{m-a+1}J_{2(a-1)} + \langle L \rangle^{m-a}J_{2a} .\]
    Because of Theorem \ref{pro:star-generators}, $I_{2,\mathcal{L}}$
    is generated by elements of the form $F_i = L/L_i$ for some
    $i=1,\ldots,s$.  So, a generator of $I_{2,\mathcal{L}}^{2a}$ has
    the form $F_{i_1}F_{i_2}\cdots F_{i_{2a}}$ where
    $i_1,\ldots,i_{2a}$ need not be distinct.  If
    $i_1 = \cdots = i_{2a} = i$, then the generator
    $F_{i_1}F_{i_2}\cdots F_{i_{2a}} = \frac{L^{2a}}{L_i^{2a}}$ of
    $I_{2,\mathcal{L}}^{2a}$ is also a generator of $J_{2a}$, so $L^{m-a}F_{i_1}F_{i_2}\cdots F_{i_{2a}} \in \langle L \rangle^{m-a}J_{2a}$.  If at least two of $i_1,\ldots,i_{2a}$ are distinct, say $i_1 \neq i_2$, then
\[F_{i_1}F_{i_2}\cdots F_{i_{2a}} = F_{i_1}L_{i_1}\frac{F_{i_2}}{L_{i_1}}F_{i_3} \cdots F_{i_{2a}}
= L\frac{F_{i_2}}{L_{i_1}}F_{i_3} \cdots F_{i_{2a}}.\]
But then 
\[
L^{m-a}F_{i_1}F_{i_2}\cdots F_{i_{2a}} = 
L^{m-a+1} \frac{F_{i_2}}{L_{i_1}}F_{i_3} \cdots F_{i_{2a}} \in 
\langle L \rangle^{m-(a-1)}J_{2(a-1)}.
\]
By induction, we then have 
$$
L^{m-a}F_{i_1}F_{i_2}\cdots F_{i_{2a}} \in \langle L \rangle^m + \langle L \rangle^{m-1}J_2   
+ \cdots + 
\langle L \rangle^{m-a+1}J_{2(a-1)} +
\langle L \rangle^{m-a}J_{2a}.
$$
This now verifies the claim.

To complete the proof, note that to form
$I_{c,\mathcal{L}}^{(2m)}$, we can add all of the generators
of $\langle L \rangle^m + \langle L \rangle^{m-1}J_2
+ \cdots +  \langle L \rangle^{1}J_{2m-2}$ to $I_{2,\mathcal{L}}^{2m}$.  
This ideal has at most
$1+s(m-1)$ minimal generators (our generating set may not be minimal)
since each ideal $J_{2a}$ has $s$ generators,
so $\sdef(I_{2,\mathcal{L}},2m) \leq 1 + s(m-1)$.
\end{proof}


\section{A geometric consequence}

By Theorem \ref{sdefect-one-classify},
if $I_{c,\mathcal{L}}$ is
a linear star configuration in
$\mathbb{P}^n$ of codimension two, then $\sdef(I_{c,\mathcal{L}},2) =1$ since $c =2$.
If $n =2$, then  the linear star configuration 
defined by $I_{c,\mathcal{L}}$ is a collection of points in $\mathbb{P}^2$,
and thus, there exist sets of points $\X$ in $\mathbb{P}^2$
with $\sdef(I_\X,2) =1$.
In general, it would be interesting to classify all the ideals $I_\X$
of sets of points $\X$ in $\mathbb{P}^2$ with $\sdef(I_\X,2) =1$.
In this section, we show under some additional hypotheses, that
if $\X$ is a set of points in $\mathbb{P}^2$ with $\sdef(I_\X,2) =1$,
then $\X$ must be a linear star configuration.

We first recall some facts about the defining ideals of points in
$\mathbb{P}^2$; many of these results are probably known to the
experts, but for completeness, we include their proofs.  Recall that
for any homogeneous ideal $I \subseteq R$, we let
$\alpha(I) = \min\{i ~|~ I_i \neq 0 \}$.  Note that for any
$m \geq 1$, $\alpha(I^m) = m \alpha(I)$.

The following is the so-called {\em Dubreil's inequality}  
(see \cite{Ca, DGM}),  but an elementary proof (which we now give) 
is also possible.

\begin{lem}\label{genslem}
Let $\mathbb{X} \subseteq \mathbb{P}^2$ be a finite set of points.
If $\alpha=\alpha(I_{\mathbb{X}})$, then $I_{\mathbb{X}}$ has 
at most $\alpha+1$ minimal generators of degree $\alpha$.
\end{lem}

\begin{proof}
Because $\alpha = \alpha(I_{\mathbb{X}})$, the Hilbert function
of $\mathbb{X}$ at $\alpha-1$ is 
${\bf H}_{R/I_\X}(\alpha-1) = \dim_\k R_{\alpha-1} = \binom{\alpha+1}{2}$.
If $I_{\mathbb{X}}$ has $d > \alpha+1$ generators of degree 
$\alpha$, then ${\bf H}_{R/I_\X}(\alpha) = \binom{\alpha+2}{2} - d <
\binom{\alpha+2}{2} - (\alpha+1) = \binom{\alpha+1}{2}$.
In other words,  ${\bf H}_{R/I_\X}(\alpha-1) > \H_{R/I_\X}(\alpha)$,
contradicting the fact that the Hilbert functions of sets of points
must be non-decreasing functions \cite[cf.~proof of Proposition 1.1 (2)]{GM}.
\end{proof}

The next lemma is a classification of those sets of points
which have exactly $\alpha+1$ minimal generators of degree $\alpha$.

\begin{lem}\label{equivconditions}
Let $\X$ be a set of points of $\mathbb{P}^2$. Then the following are 
equivalent:
\begin{enumerate}
\item[$(i)$] The ideal $I_\X$ has $\alpha+1$ minimal generators of degree
$\alpha = \alpha(I_{\mathbb{X}})$;
\item[$(ii)$] The set $\X$ is a set of $\binom{\alpha+1}{2}$ points in 
$\P^2$ having generic Hilbert function, i.e.,
\[{\bf H}_{R/I_\X}(i) = \min\{\dim_\k R_i,|\X|\} ~~~\mbox{for all $i \geq 0$; and}\]
\item[$(iii)$] The ideal $I_\X$ has a graded linear resolution.
\end{enumerate}
\end{lem}

\begin{proof} 
$(i) \Rightarrow (ii)$ If $I_\X$ has $\alpha+1$ minimal generators of degree $\alpha$, it follows that
\[\binom{\alpha+1}{2} = \H_{R/I_\X}(\alpha-1) = \H_{R/I_\X}(\alpha) = \binom{\alpha+2}{2} - 
\binom{\alpha+1}{1}.\]
Because the Hilbert function of a set of points in $\mathbb{P}^2$ is a strictly increasing
function until it reaches $|\X|$, we have $|\X| = \binom{\alpha+1}{2}$, and 
the Hilbert function of $R/I_\X$ is given by
\[\H_{R/I_\X}(t) = \min\left\{\dim_\k R_t, \binom{\alpha+1}{2}\right\}  ~~\mbox{for all $t \geq 0$}.\]

$(ii) \Rightarrow (iii)$
If $R/I_\X$ has the generic Hilbert function, one can use
Section 3 of \cite{L} to deduce that the resolution is 
\[
0 \longrightarrow  R^\alpha(-(\alpha+1)) \longrightarrow 
R^{\alpha+1}(-\alpha) \longrightarrow R  \longrightarrow  R/I_\X \longrightarrow  0,\]
i.e., the graded resolution is linear.

$(iii) \Rightarrow (i)$
Assume that $I_\X$ has a linear graded free resolution 
\[
0 \longrightarrow  R^{\beta-1}(-(\alpha+1)) \longrightarrow 
R^{\beta}(-\alpha) \longrightarrow R  \longrightarrow  R/I_\X \longrightarrow  0.\]
Since $\H_{R/I_\X}(t)=\H_{R/I_\X}(t+1)$ for $t \gg 0$,
we get that
\begin{multline*}
\dim_\k R_t-\beta \dim_\k R_{t-\alpha}+(\beta-1) \dim_\k R_{t-(\alpha+1)} \\
= \dim_\k R_{t+1}-\beta \dim_\k R_{(t+1)-\alpha}+(\beta-1)\dim_\k R_{(t+1)-(\alpha+1)}.  
\end{multline*}
This proves that $\beta=\alpha+1$, i.e., $I_\X$ has $\alpha+1$
minimal generators of degree $\alpha$.
\end{proof} 

\begin{lem}\label{degreebound}
Let $\X$ be a set of points of $\mathbb{P}^2$, and 
suppose that any of the three equivalent conditions
of Lemma \ref{equivconditions} holds.
If $I_\X^{(2)} = I_\X^2 + \langle F_1,\ldots,F_r \rangle$, i.e., 
the $F_i$ comprise a minimal set of homogeneous generators of $I_\X^{(2)}$ modulo $I_\X^2$, 
then $\deg (F_i)<2\alpha(I_\X)$ for all $i=1,\ldots,r$.
\end{lem}

\begin{proof}  We first observe that because $I_\X$ is an ideal of points,            then the saturation of $I_\X^2$ is $I_\X^{(2)}$.  If $d$ is
  the saturation degree of $I_\X^2$, i.e., the smallest integer $d$ such that $(I_\X^{(2)})_t = (I_\X^2)_t$ for all $t \geq d$, then
  it is known that ${\rm reg}(I_\X^2) \geq d$ (see, for example, the
  introduction of \cite{BG}).

  Again, because $I_\X$ is an ideal of points, we have
$$
2\reg(I_\X)\geq \reg(I_\X^2)\geq \alpha(I_\X^2) =2\alpha(I_\X)=2\reg(I_\X),
$$
where the first inequality follows from \cite[Theorem 1.1]{GGP}, and the
last equality holds from the fact that $I_\X$ has a linear resolution.
Thus, we get that
$\reg(I_\X^2)=2\alpha(I_\X)$, or in
other words, $I_\X^2$ and $I_\X^{(2)}$ agree in degrees 
$\geq \reg(I_\X^2)=2\alpha(I_\X)$. Therefore, any minimal generators of 
$I_\X^{(2)}$ have degrees less than $2\alpha(I_\X)=2\alpha$, as we wished. 
\end{proof}

When $I_\X$ is the homogeneous ideal of a finite set of points 
$\X$ in $\mathbb{P}^2$, it is well known that $I_\X$ is both perfect
and has codimension two.   In addition, $I_\X$ is a generic complete
intersection because $I_\X$ is a radical ideal in a regular 
ring and the minimal associated primes of $I_\X$ are simply
the ideals of the points $P \in \X$, and when we localize $I_\X$ at $I_P$,
we get the maximal ideal of $\k[x_0,x_1,x_2]$ localized at $I_P$,
which is a complete intersection.
We can thus apply Theorem \ref{resix2} to any homogeneous
ideal of a finite set of points in $\mathbb{P}^2$.  In particular,
we record this fact as a lemma.

\begin{lem}\label{nummingens}
Let $\X \subseteq \mathbb{P}^2$ be a finite set of points.
Suppose that $I_\X$ has $d$ minimal generators of degree
$\alpha = \alpha(I_\X)$.  Then $I_\X^2$ has $\binom{d+1}{2}$ minimal generators
of degree $\alpha(I_\X^2) = 2\alpha$.  In particular,
\[\H_{R/I_\X^2}(2\alpha) = \binom{2\alpha+2}{2} - \binom{d+1}{2}.\]
\end{lem}

\begin{proof}
If $F_1,\ldots,F_d$ are the $d$ minimal generators of degree $\alpha =
\alpha(I_\X)$,
then by Theorem \ref{resix2},  the elements of 
$\{F_iF_j ~\mid~ 1 \leq i \leq j \leq d \}$
will all be minimal generators of $I_\X^2$.  Each generator will have
degree $\alpha(I_\X^2) = 2\alpha$ and there are 
$\binom{d+1}{2}$ such generators.  For the last statement,
since $I_\X^2$ has no generators of degree $< 2\alpha$, we have
$\dim_\k (I_\X^2)_{2\alpha} = \binom{d+1}{2}$.  
\end{proof}

We also require a result
of Bocci and Chiantini.  Statement $(i)$ can be found in the
introduction of \cite{BC:1}, while $(ii)$ is \cite[Theorem 3.3]{BC:1}.

\begin{thm}\label{BCresults}
Let $\X \subseteq \mathbb{P}^2$ be a set of points. 
\begin{enumerate}
\item[$(i)$] Then  $\alpha(I_{\mathbb{X}}^{(2)}) \geq \alpha(I_{\mathbb{X}})+1$.
\item[$(ii)$] If
$\alpha(I_\X^{(2)}) = \alpha(I_\X) +1$, then $\X$ is a linear
star configuration of points or a set of colinear points.
\end{enumerate}
\end{thm}

We now come to the main result of this section.

\begin{thm}\label{partialconverse}
  Let $\X$ be a set of $\binom{\alpha+1}{2}$ points of $\mathbb{P}^2$
  with the generic Hilbert function.
  If $\sdef(I_\X,2) = 1$, then 
  $\X$ is a linear star configuration.
\end{thm}

\begin{proof}
Since $\sdef(I_\X,2) = 1$, there exits a form $F$ such that 
$I_\X^{(2)} = \langle F\rangle + I_\X^2$.
By Lemma \ref{degreebound}, ${\deg} F < 2\alpha$.

We now show that we must, in fact, have  $\deg F \leq \alpha+1$.
By Lemma \ref{equivconditions},
$I_{\mathbb{X}}$ has $\alpha+1$ generators of degree $\alpha$.
By Lemma \ref{nummingens}, the ideal $I_{\mathbb{X}}^2$ will have
$\binom{\alpha+2}{2}$ minimal generators of degree $2\alpha$.
Because $I_{\mathbb{X}}^2 \subseteq I_{\mathbb{X}}^{(2)}$, this means
\begin{eqnarray*}
\H_{R/I_{\mathbb{X}}^{(2)}}(2\alpha) &\leq& \H_{R/I_{\mathbb{X}}^2}(2\alpha)
= \binom{2\alpha+2}{2} - \binom{\alpha+2}{2}\\
& = & \frac{(2\alpha+2)(2\alpha+1) - (\alpha+2)(\alpha+1)}{2} 
=  \frac{3\alpha^2+3\alpha}{2}.
\end{eqnarray*}
  
Suppose that $\deg F > \alpha+1$.   Because $I_{\mathbb{X}}^2$ is generated 
by forms of degree $2\alpha$ or larger, we have
\[(I_{\mathbb{X}}^{(2)})_{2\alpha-1} = [\langle F\rangle + I_{\mathbb{X}}^2]_{2\alpha-1} = 
[\langle F\rangle]_{2\alpha-1},\]
and consequently, 
\[\H_{R/I_{\mathbb{X}}^{(2)}}(2\alpha-1) = \H_{R/\langle F\rangle}(2\alpha-1)
= \binom{2\alpha+1}{2} - \dim_\k \langle F\rangle_{2\alpha-1}.\]
If $\deg F = d$, then $\langle F\rangle \cong R(-d)$ as graded $R$-modules,
so $\dim_\k \langle F\rangle_{2\alpha -1} = \dim_\k R_{2\alpha-1-d} = \binom{2\alpha-d+1}{2}$.
Because $d \geq \alpha+2$, we have
\begin{eqnarray*}
\H_{R/I_{\mathbb{X}}^{(2)}}(2\alpha-1) &= &
\binom{2\alpha+1}{2} - \binom{2\alpha-d+1}{2}  \\
& \geq & \binom{2\alpha+1}{2} - \binom{2\alpha-(\alpha +2)+1}{2}  \\
& = & \frac{(2\alpha+1)(2\alpha) - (\alpha-1)(\alpha-2)}{2} \\
& = & \frac{4\alpha^2 + 2\alpha -(\alpha^2- 3\alpha +2)}{2} 
=  \frac{3\alpha^2+ 5\alpha - 2}{2}.
\end{eqnarray*}

Since $\sdef(I_\X,2) \neq 0$, $\mathbb{X}$ is not a complete intersection, and
thus $\X$ cannot be a set of points on a line.  Consequently, $\alpha \geq 2$.
But then we must have
\[  \H_{R/I_{\mathbb{X}}^{(2)}}(2\alpha-1) \geq 
\frac{3\alpha^2+ 5\alpha  - 2}{2} > 
 \frac{3\alpha^2+3\alpha}{2} \geq \H_{R/{I_\mathbb{X}^{(2)}}}(2\alpha).\]
This is a contradiction,  so $\deg F \leq \alpha+1$ as claimed.

Because ${\deg} F > \alpha$ by Theorem \ref{BCresults}, we must have
$\deg F = \alpha+1$.  Hence
$\alpha(I_{\mathbb{X}}^{(2)}) = \alpha(I_{\mathbb{X}}) +1$.  Theorem
\ref{BCresults} then implies that $\X$ is a either a linear star
configuration or a set of colinear points.  If $\X$ was a set of
colinear points, then Theorem \ref{completeintersection} would imply
that $\sdef(I_\X,2) = 0$ because colinear points are a complete
intersection.  Thus $\X$ must be a linear star configuration in
$\mathbb{P}^2$.
\end{proof}

\begin{rem}
  As we will see in Section 6, there exist sets of points $\X$ in
  $\mathbb{P}^2$ with $\sdef(I_\X,2)=1$, but $\X$ is not a linear star
  configuration.
\end{rem}

\begin{rem}
  It is natural to ask if a similar type of result holds for points in
  $\mathbb{P}^n$ with $n \geq 3$, i.e., if $\sdef(I_\X,2)=1$,
  along with some suitable hypotheses on $\X$, implies that $\X$ must
  be a linear star configuration.  However, this cannot happen.
  Indeed, if such a set of points $\X$ existed, then
  $\X = V(I_{n,\mathcal{L}})$ for some $n \geq 3$ and set of linear
  forms $\mathcal{L}$, because $\X$ is a zero-dimensional scheme.  But
  then we would have $\sdef(I_{n,\mathcal{L}},2)=1$,
  contradicting Theorem \ref{sdefect-one-classify}.
\end{rem}


\section{Application: Resolutions of squares of star configurations in $\P^n$}

In this section, we use Corollary \ref{cor:general-symb-square} to
describe a minimal free resolution of the symbolic square of the
defining ideal $I_{2,\mathcal{F}}$ of a codimension two star
configuration in $\mathbb{P}^n$.

\begin{lem}\label{L:20170926-501}
  Let $I_{2,\mathcal{F}}$ be the defining ideal of a star
  configuration of codimension two in $\mathbb{P}^n$. Assume
  $\mathcal{F}=\{F_1,\ldots,F_s\}$, where $F_1,\ldots,F_s$ are forms
  of degrees $1\leq d_1\leq \cdots \leq d_s$, and let
  $d=d_1+\cdots+d_s$. Then a graded minimal free resolution of
  $I_{2,\mathcal{F}}^2$ has the form
  \begin{equation*}
    0 \to R^{\binom{s-1}{2}}(-2d) \to \bigoplus_{1\leq i\leq s}
    R^{s-1}(-(2d-d_i)) \to \bigoplus_{1\leq i,j\leq s} R(-(2d-(d_i+d_j)))
    \to I_{2,\mathcal{F}}^2 \to 0.
  \end{equation*}
\end{lem}

\begin{proof} By \cite[Theorem 3.4]{PS:1}, the ideal
  $I_{2,\mathcal{F}}$ has a graded minimal free resolution of the form
$$
0 \to R^{s-1}(-d) \to \bigoplus _{1\leq i\leq s}R(-(d-d_i)) \to
I_{2,\mathcal{F}} \to 0.
$$
Recall that
$$
I_{2,\mathcal{F}}=\bigcap_{1\leq i < j\leq s} \langle F_i,F_j\rangle.
$$
Let $P$ be a minimal prime of $I_{2,\mathcal{F}}$ in $R$. Then $P$ has
height $2$.

\noindent {\em Claim.} There exists a unique pair $(i,j)$ such that
$\langle F_i,F_j\rangle \subseteq P$.

\noindent {\em Proof of Claim.} The existence of the pair follows from
\cite[Prop.~1.11]{AM}. Assume
$\langle F_\alpha,F_\beta\rangle \subseteq P$ for some indices
$\alpha,\beta$ with $\{\alpha,\beta\}\ne \{i,j\}$. Without loss of
generality, we may assume that $\alpha\ne i,j$. Then
$F_i,F_j, F_\alpha \in P$, which is a contradiction, since
$F_i,F_j,F_\alpha$ form a regular sequence of length 3 but $P$ has
height 2.  \hfill$\Box$

\medskip

It follows from the claim that
\begin{equation*}
  (I_{2,\mathcal{F}})_P =\bigcap_{1\leq k < l \leq s} \langle F_k,F_l\rangle_P
  = \langle F_i,F_j\rangle_P = \big\langle \tfrac{F_i}{1},\tfrac{F_j}{1}\big\rangle.
\end{equation*}
Since localization preserves regular sequences,
$(I_{2,\mathcal{F}})_P$ is a complete intersection ideal in $R_P$. We
deduce that $I_{2,\mathcal{F}}$ is a generic complete intersection
ideal.  Since $I_{2,\mathcal{F}}$ is also a perfect codimension two
ideal, we can apply Theorem
\ref{resix2} to derive the stated graded minimal free resolution of
$I_{2,\mathcal{F}}^2$.
\end{proof}

\begin{lem}\label{L:20160715-602}
  Let $I_{2,\mathcal{F}}$ be the defining ideal of a star
  configuration of codimension two in $\mathbb{P}^n$. Assume
  $\mathcal{F}=\{F_1,\ldots,F_s\}$, and set $F=F_1\cdots F_s$. Then
  \begin{enumerate}[label=(\roman*)]
  \item\label{colon-lemma-i}
    $ [I_{2,\mathcal{F}}^2:F]=\big\langle F_{i_1}\cdots F_{i_{s-2}}
    \mid 1\leq i_1 < \dots < i_{s-2} \leq s \big\rangle =I_{3,\mathcal{F}}$;
  \item\label{colon-lemma-ii}
    $I_{2,\mathcal{F}}^2 \cap \langle F\rangle = F[I_{2,\mathcal{F}}^2:F]$.
  \end{enumerate}
\end{lem}

\begin{proof} $\ref{colon-lemma-i}$ First, recall that
$$
I_{2,\mathcal{F}}^2=\bigg\langle \frac{F^2}{F_iF_j}
\,\bigg|\, 1\leq i\leq j\leq s\bigg\rangle.
$$
Given indices $1\leq i_1 < \dots < i_{s-2} \leq s$, let
$\{i_{s-1},i_s\}$ be the complement of $\{i_1, \dots, i_{s-2}\}$ in
$\{1,\ldots,s\}$. Then we have
$$
(F_{i_1}\cdots F_{i_{s-2}}) F =\frac{F^2}{F_{i_{s-1}}
  F_{i_{s}}} \in I_{2,\mathcal{F}}^2,
$$
and so $F_{i_1}\cdots F_{i_{s-2}}\in [I_{2,\mathcal{F}}^2:F]$.

Conversely, let $G\in [I_{2,\mathcal{F}}^2 : F]$. Since
$G F \in I_{2,\mathcal{F}}^2$, we have that
\begin{equation}
  \label{EQ:20160715-101}
  G F = \sum_{1\leq i\leq s} A_i \frac{F^2}{F_i^2}
  + \sum_{1\leq i< j\leq s} B_{i,j} \frac{F^2}{F_iF_j}
\end{equation}
for some $A_i, B_{i,j}\in R$.

\medskip

\noindent
{\em Claim.} For every $1\leq i \leq s$, $F_i$ divides $A_i$.

\medskip

\noindent {\em Proof of Claim.} For $i=1$,
$$
G F=A_1 \frac{F^2}{F_1^2}+ \sum_{2\leq i\leq
  s} A_i \frac{F^2}{F_i^2}+ \sum_{1\leq i< j\leq s}
B_{i,j} \frac{F^2}{F_iF_j}.
$$
Hence
\begin{equation*}
  G F-
  \sum_{1\leq i< j\leq s} B_{i,j} \frac{F^2}{F_iF_j} - 
  \sum_{2\leq i\leq s} A_i \frac{F^2}{F_i^2} = A_1  \frac{F^2}{F_1^2}.
\end{equation*}
For all $h\neq 1$, $F_1, F_h$ is, by assumption, a regular
sequence. This implies that $F_1$ and $F_h$ are coprime.  Therefore
$F_1$ must divide $A_1$ because $F_1$ divides every
term on the left hand side.  Similarly, one can show that $F_i$ divides
$A_i$ for all $1\leq i\leq s$.  \hfill$\Box$

Let $A_i=F_i A_i'$ for some $A_i'\in R$. We can rewrite
equation~\eqref{EQ:20160715-101} as
$$
G F= \sum_{1\leq i\leq s} A_i' \frac{F^2}{F_i}+ \sum_{1\leq i< j\leq s}
B_{i,j} \frac{F^2}{F_iF_j}.
$$
Dividing both sides by $F$, we obtain
$$
G=\sum_{1\leq i\leq s} A_i' \frac{F}{F_i}+ \sum_{1\leq i<
  j\leq s} B_{i,j} \frac{F}{F_iF_j}
$$
proving that $G$ is in
$\big\langle F_{i_1}\cdots F_{i_{s-2}} \mid 1\leq i_1 < \dots <
i_{s-2} \leq s \big\rangle$.

\noindent
$\ref{colon-lemma-ii}$ If $G \in I_{2,\mathcal{F}}^2 \cap \langle F\rangle$, then
$G = G'F \in I_{2,\mathcal{F}}^2$.  So
$G' \in [I_{2,\mathcal{F}}^2:F]$, and thus
$G = FG' \in F[I_{2,\mathcal{F}}^2:F]$.
Conversely, if $H \in F[I_{2,\mathcal{F}}^2:F]$, we have $H = FH'$
with $H' \in [I_{2,\mathcal{F}}^2:F]$.  It is then immediate that
$H \in I_{2,\mathcal{F}}^2 \cap \langle F\rangle$, which completes the proof of
this lemma.
\end{proof} 

\begin{thm}\label{T:20160710-802}
  Let $I_{2,\mathcal{F}}$ be the defining ideal of a star
  configuration of codimension two in $\mathbb{P}^n$. Assume
  $\mathcal{F}=\{F_1,\ldots,F_s\}$, where $F_1,\ldots,F_s$ are forms
  of degrees $1\leq d_1\leq \cdots \leq d_s$, and let
  $d=d_1+\cdots+d_s$. Then a graded minimal free resolution of
  $I_{2,\mathcal{F}}^{(2)}$ has the form
  \begin{equation*}
    0 \to \bigoplus_{1\leq i\leq s} R(-(2d-d_i)) \to \left(
      \bigoplus_{1\leq i\leq s}R(-(2d-2d_i)) \right) \oplus
    R(-d)\to I_{2,\mathcal{F}}^{(2)} \to 0.
  \end{equation*}
\end{thm}

\begin{proof}
  Let $F = F_1\cdots F_s$.  Thanks to Corollary
  \ref{cor:general-symb-square}, there is a short exact sequence
  \begin{equation*}
    0 \to I_{2,\mathcal{F}}^2 \cap \langle F\rangle
    \to I_{2,\mathcal{F}}^2 \oplus \langle F\rangle
    \to I_{2,\mathcal{F}}^{(2)} \to 0.
  \end{equation*}

  We proceed to describe a minimal free resolution of the left term.
  By Lemma~\ref{L:20160715-602} $\ref{colon-lemma-i}$,  $[I_{2,\mathcal{F}}:F] = I_{3,\mathcal{F}}$. By \cite[Theorem 3.4]{PS:1}, a minimal free resolution of $I_{3,\mathcal{F}}$ has the form
  \begin{equation*}
    0 \to R^{\binom{s-1}{2}}(-d) \to
    {\ds\bigoplus_{1\leq i\leq s}}R^{s-2}(-(d-d_i))\to 
    {\ds\bigoplus_{1\leq i<j\leq s} R(-(d-(d_i+d_j))}.
  \end{equation*}
  By Lemma~\ref{L:20160715-602} $\ref{colon-lemma-ii}$, we have
  $I_{2,\mathcal{F}}^2 \cap \langle F\rangle=F[I_{2,\mathcal{F}}^2:F]
  = F I_{3,\mathcal{F}}$.  Since $F$ has degree $d$, to obtain a
  minimal free resolution of $F I_{3,\mathcal{F}}$ it is enough to add
  $d$ to the degrees of the generators of the free modules in the
  resolution above. More explicitly, a minimal free resolution of
  $I_{2,\mathcal{F}}^2 \cap \langle F\rangle$ has the form
  \begin{equation*}
    0 \to R^{\binom{s-1}{2}}(-2d) \to
    {\bigoplus_{1\leq i\leq s}}R^{s-2}(-(2d-d_i))\to 
    {\bigoplus_{1\leq i<j\leq s} R(-(2d-(d_i+d_j))}.
  \end{equation*}

  Next we describe a minimal free resolution of the middle term.
  We found a minimal free resolution for $I^2_{2,\mathcal{F}}$ in Lemma~\ref{L:20170926-501}. Since $0\to R(-d) \to \langle F\rangle \to 0$ is a minimal free resolution of $\langle F\rangle$, we can take a direct sum of the resolutions of the two ideals to obtain the complex
  \begin{equation*}
    0\to R^{\binom{s-1}{2}}(-2d)
    \to \bigoplus_{1\leq i\leq s}R^{s-2}(-(2d-d_i)) \to
    \left(\bigoplus_{1\leq i,j\leq s} R(-(2d-(d_i+d_j)))\right) \oplus R(-d),
  \end{equation*}
  which is a minimal free resolution of $I_{2,\mathcal{F}}^2 \oplus \langle F\rangle$.

  Our goal is to describe a minimal free resolution of the right term in the short exact sequence. Using a mapping cone construction \cite{HS}, we obtain a free resolution of $I_{2,\mathcal{F}}^{(2)}$ of the form
  \begin{multline*}
    0 \to R^{\binom{s-1}{2}}(-2d) \to
    \left(\bigoplus_{1\leq i\leq s}R^{s-2}(-(2d-d_i)) \right)
    \oplus R^{\binom{s-1}{2}}(-2d) \to\\
    \to \left(\bigoplus_{1\leq i<j\leq s} R(-(2d-(d_i+d_j)))\right) \oplus
    \left( \bigoplus_{1\leq i\leq s}R^{s-1}(-(2d-d_i))\right) \to\\
    \to \left(\bigoplus_{1\leq i,j\leq s} R(-(2d-(d_i+d_j)))\right) \oplus R(-d).
  \end{multline*}

  The ideal $I_{2,\mathcal{F}}^{(2)}$ is Cohen-Macaulay by
  \cite[Corollary 3.7]{GHMN}.  In particular, a graded minimal free
  resolution of $I_{2,\mathcal{F}}^{(2)}$ has length 1.  Hence the
  $R^{\binom{s-1}{2}}(-2d)$ at the end of the resolution must cancel
  out the $R^{\binom{s-1}{2}}(-2d)$ in the penultimate module.  In
  addition, $\bigoplus_{1\leq i\leq s} R^{s-2}(-(2d-d_i))$ must cancel
  with part of $\bigoplus_{1\leq i\leq s} R^{s-1}(-(2d-d_i))$ in
  homological degree two.  After these cancellations, we are left with
  the smaller resolution
  \begin{multline*}
    0\to \left(\bigoplus_{1\leq i<j\leq s} R(-(2d-(d_i+d_j)))\right)
    \oplus
    \left(\bigoplus_{1\leq i\leq s}R(-(2d-d_i)) \right) \to\\
    \to \left(\bigoplus_{1\leq i,j\leq s} R(-(2d-(d_i+d_j)))\right)
    \oplus R(-d)
  \end{multline*}
  of $I_{2,\mathcal{F}}^{(2)}$.
  By Corollary \ref{cor:general-symb-square},
  $I_{2,\mathcal{F}}^{(2)}$ has exactly one generator of degree $d$,
  namely $F$, and the rest of the generators have degrees
  $2d-(d_i+d_j)$.  The generators of degree $2d-(d_i+d_j)$ with
  $d_i\neq d_j$ are redundant since they are multiples of $F$.  Each
  redundant generator gives rise to a relation of degree
  $2d-(d_i+d_j)$ that expresses the redundant generator in terms of
  the minimal ones. As such, we can remove these redundant generators
  along with the corresponding relations.  We are then left with
  \begin{equation*}
    0 \to \bigoplus_{1\leq i\leq s} R(-(2d-d_i)) \to \left(
      \bigoplus_{1\leq i\leq s}R(-(2d-2d_i)) \right) \oplus
    R(-d)\to I_{2,\mathcal{F}}^{(2)} \to 0.
  \end{equation*}
  This resolution must now be minimal since no further cancellation is
  possible.
\end{proof}

\begin{rem}\label{generalize}
If $I_{2,\mathcal{L}}$ defines a linear star configuration in $\mathbb{P}^n$,
our formula agrees with the formula of \cite[Theorem 3.2]{GHM} with
$c=2$; thus Theorem \ref{T:20160710-802} is a generalization of \cite{GHM}
in the sense that the star configuration
need not be linear.
\end{rem}


\section{General sets of points}

In this section, we study general sets $\X$ of
points in $\mathbb{P}^2$.  Specifically, 
we  characterize when $\sdef(I_\X,2) = 1$.

Recall that a property holds for {\em a general set} of $s$ points in $\P^n$ 
if the subset of $(\P^n)^s$ for which it holds contains a nonempty open subset.
%
If $\mathbb{X} \subseteq \mathbb{P}^n$ is a general set of points, then $\mathbb{X}$ has the 
{\it generic Hilbert function}, 
that is,
\[
\H_{R/I_\mathbb{X}}(i) = \min\{\dim_\k R_i, |\mathbb{X}|\} 
~~\mbox{for all $i \geq 0$}.
\]

The key ingredient that we require is the following
famous result of Alexander and Hirschowitz which
computes the Hilbert function of $R/I_{\mathbb{X}}^{(2)}$ 
when $\mathbb{X}$ is a set of general points in 
$\mathbb{P}^n$ (we have specialized their
result to  $\mathbb{P}^2$).  Roughly speaking, except
if $s =2$ or $5$, the Hilbert function of $R/I_{\mathbb{X}}^{(2)}$ is the 
generic Hilbert function of $3|\mathbb{X}|$ points.

\begin{thm}[{\cite[Theorem 2]{AH:1}}]
\label{AH}
Let $\mathbb{X}$ be a set of $s$ general points 
in $\mathbb{P}^2$.  If $s \neq 2,5$, then
\[\H_{R/I_{\mathbb{X}}^{(2)}}(i) = \min\{\dim_\k R_i, 3s \} ~~\mbox{for
all $i \geq 0$.}\]
If $s = 5$, then 
\[\H_{R/I_{\mathbb{X}}^{(2)}}(i) = 
\begin{cases}
\min\{\dim_\k R_i, 3s \} & i \neq 4 \\
14 & i = 4. 
\end{cases}
\]
\end{thm}

In fact, the graded minimal free resolution of 
$I_\X$ and $I_\X^{(2)}$ for $s$ general
points in $\mathbb{P}^2$ is known. The resolution of $I_\X$ and $I_\X^{(2)}$
is the cumulative work of many people.
For sets $\X$ of simple points, the minimal resolution of $I_\X$
was  worked out by Geramita and Maroscia \cite{GM}, Geramita,
Gregory, and Roberts \cite{GGR}, and Lorenzini \cite{L}.

For $I^{(2)}_\X$, Catalisano's work \cite{C} determines the resolution
of $I^{(2)}_\X$ for $s\le 5$, while Id\`a \cite{I} handles the case of $s>5$ 
(thereby recovering known results when $s\le 9$, and proving a conjecture 
of Harbourne \cite[Conjecture 6.3]{H1} in the special case of $m=2$).
We record only the consequences we need.

%
%

\begin{lem}\label{specialcases}
Let $\X$ be a set of $s$ general points in $\mathbb{P}^2$.
\begin{enumerate}
\item[$(i)$] If $s = 5$, then the graded minimal free resolution
of $I_\X$, respectively $I_\X^{(2)}$, is
\[0 \longrightarrow R^2(-4) \longrightarrow R(-2) \oplus R^2(-3)
\longrightarrow I_{\mathbb{X}} \longrightarrow 0, ~~\mbox{respectively}\]
\[0 \rightarrow  R^2(-6) \oplus R(-7) \rightarrow R(-4) \oplus R^3(-5)
\rightarrow I_\X^{(2)} \rightarrow 0.\]
\item[$(ii)$] 
If $s = 7$, then the graded minimal free resolution
of $I_\X$, respectively $I_\X^{(2)}$, is
\[0 \rightarrow R(-4) \oplus R(-5) \rightarrow R^3(-3) \rightarrow I_\X
\rightarrow 0, ~~\mbox{respectively}\]
\[0 \rightarrow R^6(-7) \rightarrow R^7(-6) \rightarrow I_\X^{(2)} \rightarrow 0.\]
\item[$(iii)$] 
If $s = 8$, then the graded minimal free resolution
of $I_\X$, respectively $I_\X^{(2)}$, is
\[0 \rightarrow R^2(-5) \rightarrow R^2(-3) \oplus R(-4) \rightarrow I_\X
\rightarrow 0, ~~\mbox{respectively}\]
\[0 \rightarrow R^3(-8) \rightarrow R^4(-6) \rightarrow  I_\X^{(2)} \rightarrow 0.\]
\item[$(iv)$] 
If $s = 9$, then the graded minimal free resolution
of $I_\X$, respectively $I_\X^{(2)}$, is
\[0 \rightarrow R^3(-5) \rightarrow R(-3)\oplus R^3(-4)
\rightarrow  I_\X
\rightarrow 0, ~~\mbox{respectively}\]
\[0 \rightarrow R^6(-8) \rightarrow R(-6) \oplus R^6(-7) 
\rightarrow I_{\mathbb{X}}^{(2)} \rightarrow 0.\]
\end{enumerate}
\end{lem}

We now present the main result of this section.

\begin{thm}\label{generalpoints}
Let $\X$ be a set of $s$ general points in $\mathbb{P}^2$.
Then 
\begin{enumerate}
\item[$(i)$] $\sdef(I_\X,2) = 0$ if and only if $s = 1,2$ or 
$4$.
\item[$(ii)$] $\sdef(I_\X,2) =1$ if and only if $s =3, 5, 7$,
or $8$.
\item[$(iii)$] $\sdef(I_\X,2) > 1$ if and only if $s=6$ or $s \geq 9$.
\end{enumerate}
\end{thm}

\begin{proof}
By Theorem \ref{completeintersection}, $\sdef(I_\X,2) = 0$ if and 
only if $\X$ is a complete intersection.  But a set of $s$
general points is a complete intersection if and only if $s =1,2,$ or $4$
(e.g., \cite[Exercise 11.9]{Harris}).  This proves $(i)$.

We next consider the special cases of $s=3,5,6,7,8,9$.

If $s = 3$, then $\mathbb{X}$ is also a linear star configuration.
Indeed, for each pair of points $P_i,P_j$ with $i \neq j$, take the unique
line $L_{i,j}$ through those two points.  Then $I_\X = I_{2,\mathcal{L}}$ where
$\mathcal{L} = \{L_{1,2},L_{1,3},L_{2,3}\}$.
Then $\sdef(I_\X,2) = 1$ by Theorem 
\ref{sdefect-one-classify}.

For the cases $s = 5,6,7,$ and $9$, we first observe that
\[\dim_\k(I_\mathbb{X}^{(2)}/I_\mathbb{X}^2)_{\alpha(I_\mathbb{X}^{(2)}/I_\mathbb{X})}
\leq \sdef(I_\X,2) \leq \sum_{t \geq 0} \dim_\k (I_\mathbb{X}^{(2)}/I_\mathbb{X}^2)_t.
\]
We can use Theorem \ref{resix2} and Lemma \ref{specialcases} to
find the Hilbert functions of $R/I_{\X}^{(2)}$ and $R/I_\X^2$ for
$s = 5,6,7,$ and $9$.  In these four cases, we will find 
that $\dim_\k (I_\mathbb{X}^{(2)})_t = \dim_\k(I_\mathbb{X}^2)_t = 0$ if $t \neq 
\alpha(I_\mathbb{X}^{(2)}/I_\mathbb{X})$, and consequently the 
above inequalities give
 \[\dim_\k(I_\mathbb{X}^{(2)}/I_\mathbb{X}^2)_{\alpha(I_\mathbb{X}^{(2)}/I_\mathbb{X})}
= \sdef(I_\X,2).\] 
Furthermore, we can use these Hilbert functions to compute
the symbolic defect;  precisely,
\[\sdef(I_\X,2) = \begin{cases}
1 & \mbox{if $s = 5,7$} \\
3 & \mbox{if $s = 6,9$}.
\end{cases}
\]

When $s = 8$, the Hilbert functions of $I_\X^{(2)}$ and $I_\X^{2}$ disagree
in two degrees, so the above approach does not work.
Instead, if $s =8$, then Lemma \ref{specialcases} $(iii)$ implies that
$\alpha(I_\X) = 3$ and $I_\X$ has two minimal generators
of degree 3.  So, $I_\X^2$ has three minimal generators of degree 6.
By Lemma \ref{specialcases} $(iii)$, 
$\alpha(I_{\X}^{(2)}) = 6$ and $I_\X^{(2)}$ has four minimal generators
of degree 6.
So, there exists a form $F \in (I_\X^{(2)})_6 \setminus (I_\X^2)_6$.
But since $I_\X^{(2)}$ is generated by
these four generators of degree 6, 
$I_{\mathbb{X}}^{(2)} = \langle F \rangle + I_{\mathbb{X}}^2$,
that is, $\sdef(I_\X,2) = 1$.


Going forward, we now assume that $s \geq 10$. Our goal is to
show that $\sdef(I_\X,2) > 1$.  To do this, we first will show
that if $\sdef(I_X,2) =1$ and $F$ is any homogeneous form 
such that $I_\X^{(2)} = \langle F \rangle + I_\X^2$, then
the degree of $F$ is restricted.
Below, $\alpha = \alpha(I_\X)$.

\noindent
{\it Claim.}  If $s \geq 10$ and $I_\X^{(2)} = \langle F \rangle + I_\X^2$, 
then $\deg{F} \geq 2\alpha -1$.

\medskip

\noindent
{\it Proof of Claim.}  Suppose that $d = \deg F \leq 2\alpha -2$.
Because $s \geq 10$, 
\[\H_{R/I_{\mathbb{X}}^{(2)}}(d) = \H_{R/I_{\mathbb{X}}^{(2)}}(d+1) = 3|\mathbb{X}|
~~\mbox{by Theorem \ref{AH}}.\]
On the other hand, since $I_{\mathbb{X}}^{(2)} = \langle F \rangle  
+ I_{\mathbb{X}}^2$ and $\deg F \leq 2\alpha-2$, we have
\[\dim_\k (I_{\mathbb{X}}^{(2)})_d = \dim_\k \langle F \rangle_d  = 1  ~~~\mbox{and} 
~~\dim_\k (I_{\mathbb{X}}^{(2)})_{d+1} = \dim_\k \langle F \rangle_{d+1}  = 3.\]
But this then means that 
\begin{eqnarray*}
\binom{d+2}{2} - 1 &= & \H_{R/I_{\mathbb{X}}^{(2)}}(d) = 3|\mathbb{X}| 
=  \H_{R/I_{\mathbb{X}}^{(2)}}(d+1) = \binom{d+3}{2} - 3.
\end{eqnarray*}
So $d$, the degree of $F$, would have to satisfy
\[
\binom{d+2}{2} - \binom{d+3}{2} +2  = 0 \Leftrightarrow d = 0.
\]
But $\deg F > 0$.  So, $\deg F \geq 2\alpha-1$. \hfill$\Box$

Now suppose that $s \geq 10$ and $\sdef(I_\X,2) = 1$.
Consequently, there is a homogeneous form $F$ 
such that $I_\X^{(2)} = \langle F \rangle + I_\X^2$,
and furthermore,  by the above claim, $\deg F \geq 2\alpha-1$ 
where $\alpha = \alpha(I_\X)$.
We now consider the cases $\deg F = 2\alpha-1$,
and $\deg F \geq 2\alpha$ separately.

\noindent
{\it Case 1.} $I_\X^{(2)} = \langle F \rangle  + I_\X^2$ 
with  $\deg F = 2\alpha -1$.

\noindent
If $\deg F = 2\alpha-1$, then we first claim that $\alpha \leq 11$.
Indeed, by Theorem \ref{AH} 
\[\H_{R/I_{\mathbb{X}}^{(2)}}(2\alpha-2) = \binom{2\alpha}{2} \leq 3|\mathbb{X}|
= \H_{R/I_{\mathbb{X}}^{(2)}}(2\alpha-1) = \binom{2\alpha+1}{2} -1\]
where the last equality follows from the fact that $I_{\mathbb{X}}^{(2)}$
has exactly one generator of degree $2\alpha-1$.  On the other hand,
we know that $|\mathbb{X}| < \binom{\alpha+2}{2}$ since 
$s$ general points have the generic Hilbert function, 
so $\alpha$ is by definition the smallest number $i$ such that $\binom{i+2}{2} >
s = |\mathbb{X}|$.  Combining these inequalities, we have
\[\binom{2\alpha}{2} \leq 3|\mathbb{X}| <  3\binom{\alpha+2}{2}.\]
or equivalently, $\alpha$ must satisfy
\[
\frac{2\alpha(2\alpha-1)}{2}-\frac{3(\alpha+2)(\alpha+1)}{2} < 0 
\Leftrightarrow \alpha^2 -11\alpha-6 < 0.
\]
But the last inequality only holds if $\alpha \leq 11$.  
Since we are also assuming that  $|\X| \geq 10$,
we have $4 \leq \alpha \leq 11$.

As we noted above, if $\deg F = 2\alpha-1$,
then $3|\mathbb{X}| = \binom{2\alpha+1}{2}-1$ must be also be satisfied.
Via a direct calculation, we see that $\binom{2\alpha+1}{2}-1$
is divisible by $3$ with $4 \leq \alpha \leq 11$ if and only
if $\alpha = 5, 8, 11$.


So, if $3|\X| = \binom{2\alpha+1}{2}-1$ with $4 \leq \alpha \leq 11$,
then we have $|\X| = 54/3 = 18$, or $|\X| = 135/3 = 45$, or $|\X| = 252/3 = 84$.
In all other cases, we cannot have  
$I_{\mathbb{X}}^{(2)} = \langle F \rangle + I_{\mathbb{X}}^2$.  

However, if $|\X| = 45$, then $\alpha(I_\X) = 9$, not $8$.  
Also, if $|\X| = 84$, then $\alpha(I_\X) = 12$, not $11$.  If
$|\X| = 18$, then $\alpha(I_\X) = 5$.  So we need
a separate argument to show  that  $I_{\mathbb{X}}^{(2)} \neq \langle F \rangle  + I_{\mathbb{X}}^2$.  

So, let $s = 18$ with $\alpha = 5$ and  suppose that
$I_{\mathbb{X}}^{(2)} = \langle F \rangle + I_{\mathbb{X}}^2$ with $\deg{F} = 9$.
Then $(I_{\mathbb{X}}^{(2)})_{10} = (F+ I_{\mathbb{X}}^2)_{10}$.  Now
by Theorem \ref{AH}, $\dim_\k (I_{\mathbb{X}}^{(2)})_{10} = 12$.
On the other hand, $I_{\mathbb{X}}$ has three generators of degree $\alpha = 5$,
so by Lemma \ref{nummingens}, $I_{\mathbb{X}}^2$ has six generators of 
degree $2\alpha = 10$, and no smaller generators.  So
$\dim_\k (\langle F \rangle + I_{\mathbb{X}}^2)_{10} \leq \dim_\k (\langle F \rangle)_{10} + 
\dim_\k (I_{\mathbb{X}}^2)_{10} = 3+6 = 9$.  So, by a dimension count, we cannot
have $I_{\mathbb{X}}^{(2)} = \langle F \rangle + I_{\mathbb{X}}^2$.  

To summarize this case, if $s \geq 10$, there is no set 
of $s$ general points with 
$I_\X^{(2)} = \langle F \rangle + I_\X^2$ with $\deg F = 2\alpha-1$. 

\noindent
{\it Case 2.} $I_\X^{(2)} = \langle F \rangle  + I_\X^2$ with  $\deg F \geq 2\alpha$.

\noindent
If $\deg F \geq 2\alpha$, then we claim that $\alpha \leq 7$.
Indeed, since $I_{\mathbb{X}}^{(2)}$ will be generated by forms of
degree $2\alpha$ or larger, we have
\[\H_{R/I_{\mathbb{X}}^{(2)}}(2\alpha-1) = \binom{2\alpha+1}{2} \leq 3|\mathbb{X}|. \]
On the other hand, $|\mathbb{X}| < \binom{\alpha+2}{2}$.  Combining these
two inequalities gives
\[\binom{2\alpha+1}{2} \leq 3|\mathbb{X}| < 3\binom{\alpha+2}{2}.\]
So, $\alpha$ must satisfy
\[(2\alpha+1)(2\alpha) < 3(\alpha+2)(\alpha+1) \Leftrightarrow \alpha^2-7\alpha-6 < 0 \Leftrightarrow \alpha \leq 7.\]
Moreover, because $s \geq 10$, we have $4 \leq \alpha \leq 7$, 
or equivalently, $10 \leq s = |\mathbb{X}| \leq 35$. 

Let $d = \binom{\alpha+2}{2} - |\mathbb{X}|$, that is, $d$ is the number 
of minimal generators of $I_{\mathbb{X}}$ of degree $\alpha$.  
If $\deg F = 2\alpha$  and $I_{\mathbb{X}}^{(2)} = \langle F \rangle + I_{\mathbb{X}}^2$,
then $I_{\mathbb{X}}^{(2)}$ has $\binom{d+1}{2}+1$ minimal generators of degree $2\alpha$.
If $\deg F > 2\alpha$ and  $I_{\mathbb{X}}^{(2)} = \langle F \rangle  
+ I_{\mathbb{X}}^2$, then
$I_{\mathbb{X}}^{(2)}$ has $\binom{d+1}{2}$ minimal generators of degree $2\alpha$.
So, we will have
\[\H_{R/I_{\mathbb{X}}^{(2)}}(2\alpha) = 3|\mathbb{X}| =
\left\{
\begin{array}{ll}
\ds\binom{2\alpha+2}{2} -\binom{d+1}{2}- 1 & \mbox{if $\deg F = 2\alpha$},\\[2ex]  
\ds\binom{2\alpha+2}{2} -\binom{d+1}{2} & \mbox{if $\deg F > 2\alpha$}.
\end{array}
\right.
\]

Thus, to summarize,
if $I_{\mathbb{X}}^{(2)} = \langle F \rangle  + I_{\mathbb{X}}^2$ with
$\deg F \geq 2\alpha$, then
\begin{enumerate}
\item[$(a)$] $10 \leq |\mathbb{X}| \leq 35$,
\item[$(b)$] $\binom{2\alpha+1}{2} \leq 3|\mathbb{X}| < \binom{2\alpha+2}{2}$,
and  
\item[$(c)$] either 
$3|\mathbb{X}| =
\binom{2\alpha+2}{2} -\binom{d+1}{2}- 1$ or $3|\mathbb{X}| =
\binom{2\alpha+2}{2} -\binom{d+1}{2}$ must hold with 
$d = \binom{\alpha+2}{2} - |\mathbb{X}|$.
\end{enumerate}
A direct computation for each value $10 \leq |\X| \leq 35$ shows
that no value of $|\X|$ satisfies both of $(b)$ and $(c)$.  
Table \ref{table2} explicitly verifies this statement;
note that in the table, (T) denotes true
and (F) denotes false.

\footnotesize
\begin{table}[h!]
\def\arraystretch{1.2}
\begin{tabular}{|c|c|c|c|c|c|}
\hline
$|\mathbb{X}|$ & $\alpha = \alpha(I_{\mathbb{X}})$ & 
$\binom{2\alpha+1}{2} \leq 3|\mathbb{X}| < \binom{2\alpha+2}{2}$ & $d$ & $\binom{2\alpha+2}{2} -\binom{d+1}{2}- 1
= 3|\mathbb{X}|$ 
& $\binom{2\alpha+2}{2} -\binom{d+1}{2} = 3|\mathbb{X}|$\\
\hline
\hline
10 & 4 & $36 \leq 30 < 45$ (F) & & & \\
11 & 4 & $36 \leq 33 < 45$ (F) & & & \\
12 & 4 & $36 \leq 36 < 45$ (T) & 3& $39 = 36$ (F)& $38 =36$ (F)\\
13 & 4 & $36 \leq 39 < 45$ (T) & 2& $41 = 39$ (F)& $42 =39$ (F)  \\
14 & 4 & $36 \leq 42 < 45$ (T) & 1& $43 = 42$ (F)& $44 =42$ (F)\\
\hline
\hline
15 & 5 & $55 \leq 45 < 66$ (F) & & & \\
16 & 5 & $55 \leq 48 < 66$ (F) & & & \\
17 & 5 & $55 \leq 51 < 66$ (F) & & & \\
18 & 5 & $55 \leq 54 < 66$ (F) &  & & \\
19 & 5 & $55 \leq 57 < 66$ (T) & 2& $62 =57$ (F)& $63=57$ (F)\\
20 & 5 & $55 \leq 60 < 66$ (T) & 1& $64 =60$ (F)& $65=60$ (F)\\
\hline
\hline
21 & 6 & $78 \leq 63 < 91$ (F) & & & \\
22 & 6 & $78 \leq 66 < 91$ (F) & & & \\
23 & 6 & $78 \leq 69 < 91$ (F) & & & \\
24 & 6 & $78 \leq 72 < 91$ (F) & & & \\
25 & 6 & $78 \leq 75 < 91$ (F) & & & \\
26 & 6 & $78 \leq 78 < 91$ (T) &2 &$87=78$ (F) &$88 =78$ (F) \\
27 & 6 & $78 \leq 81 < 91$ (T) &1 &$89=81$ (F) &$90=81$ (F) \\
\hline
\hline
28 & 7 & $105 \leq 84 < 120$ (F) & & & \\
29 & 7 & $105 \leq 87 < 120$ (F) & & & \\
30 & 7 & $105 \leq 90 < 120$ (F) & & & \\
31 & 7 & $105 \leq 93 < 120$ (F) & & & \\
32 & 7 & $105 \leq 96 < 120$ (F) & & & \\
33 & 7 & $105\leq  99 < 120$ (F) & & & \\
34 & 7 & $105 \leq 102 < 120$ (F) & & & \\
35 & 7 & $105\leq 105 < 120$ (T) &1 &$118=105$ (F) &$119=105$ (F) \\
\hline
\end{tabular}
\\
\caption{Comparing inequalities and equalities with $\deg F  \geq 2\alpha$}\label{table2}
\end{table}
\normalsize

To summarize this case, if $s \geq 10$, there is no set 
of $s$ general points with 
$I_\X^{(2)} = \langle F \rangle + I_\X^2$ with $\deg F \geq  2\alpha$.  Thus
combining this case with the previous case, we see that
if $s \geq 10$, then $\sdef(I_X,2) > 1$, 
thus completing the proof.
\end{proof}

\begin{rem}
The special case $s=6$ in the Theorem \ref{generalpoints} can also
be explained by appealing to Theorem \ref{partialconverse}.  The ideal
$I_\X$ of six general points in $\mathbb{P}^2$ has a linear resolution.
So, if $\sdef(I_\X,2) =1$, then
the six points must be a linear star configuration by Theorem \ref{partialconverse}, and in particular,
three of the six points must be on the same line.  But
six general points is not a star configuration since three of
the six points cannot lie on a line.
\end{rem}

\begin{exmp}\label{8points}
As mentioned in the introduction, there are many questions
one can ask about the symbolic defect sequence.  We end this
section with an example to show that the symbolic defect sequence
need not be a non-decreasing sequence.  Consider the ideal
$I_\X$ when $\X$ is  eight general
points in $\mathbb{P}^2$.  Using \texttt{Macaulay2} \cite{Mt},
we found that the symbolic defect sequence $\{\sdef(I_\X,m)\}_{m=0}^{\infty}$ 
begins
\[0, 1, 3, 6, 10, 9, 7\]
and thus, the symbolic defect sequence can decrease. Understanding the 
long term behavior of this sequence would be of interest.
\end{exmp}


\end{document}